\documentclass[a4paper,12pt]{amsart}
\usepackage{tikz}
\usepackage{fullpage}
\usepackage{enumerate}
\usepackage{amsmath,amscd}
\usepackage{upgreek} 

\usepackage{amsthm, stmaryrd}
\usepackage{hyperref}
\usetikzlibrary{3d}
\usepackage{amssymb,MnSymbol}
\usepackage{latexsym}

\newtheorem{Theorem}{Theorem}

\newtheorem*{Addendum}{Addendum to Theorem~1}

\newtheorem{Proposition}[Theorem]{Proposition}
\newtheorem{Lemma}[Theorem]{Lemma}
\newtheorem{Corollary}[Theorem]{Corollary}

\theoremstyle{definition}
\newtheorem{Definition}[Theorem]{Definition}
\newtheorem{Example}[Theorem]{Example}
\newtheorem{Remark}[Theorem]{Remark}

\def\CC{{\mathbb C}}

\def\ZZ{{\mathbb Z}}

\def\NN{{\mathbb N}}

\def\sll{\mathfrak{sl}}
\def\sl2{\sll_2(\CC)}

\def\SLL{\operatorname{SL}}
\def\SL2{\SLL_2(\CC)}

\def\SU2{\operatorname{SU}(2)}

\def\tr{\operatorname{tr}}

\def\Ad{\operatorname{Ad}}

\def\id{\mathrm{id}}

\newcommand{\ns}[1]{#1_{ns}}

\def\b1{k}

\def\co{\,\colon\,}

\makeatletter
\@namedef{subjclassname@2020}{\textup{2020} Mathematics Subject Classification}
\makeatother


\title{The scheme of characters in $\mathrm{SL}_2$}

\author{Michael Heusener}
\address{Universit\'e Clermont Auvergne, CNRS, Laboratoire de Math\'ematiques Blaise Pascal, F-63000 Clermont-Ferrand, France}
\email{michael.heusener@uca.fr.}

\author{Joan Porti} 
\thanks{Both authors partially supported by   
FEDER-AEI (grant number PID2021-125625NB-100 and 
Mar\'\i a de Maeztu Program CEX2020-001084-M) and by the Laboratoire de 
Math\'{e}matiques Blaise Pascal of the Universit\'{e} Clermont Auvergne}
\address{ Departament de Matem\`atiques, Universitat Aut\`onoma de Barcelona, 
08193 Cerdanyola del Vall\`es, Spain, and 
Centre de Recerca Matem\`atica (CRM) }
\email{porti@mat.uab.cat}

\date{\today}

\begin{document}
\begin{abstract} 
The aim of this article is to study the $\SL2$--character scheme of a finitely 
generated group.
Given a presentation of a finitely generated group $\Gamma$, we give equations defining the coordinate ring of
the scheme of $\SL2$--characters of $\Gamma$
(finitely many equations when $\Gamma$ is finitely presented).
We also study the scheme of abelian and non-simple representations and characters.
Finally we apply our results to study the $\SL2$--character scheme of the Borromean rings.

\end{abstract}
\maketitle

\section{Introduction}

For a finitely 
generated group $\Gamma$, we fix a  presentation
\begin{equation}\label{eqn:fpGroup}
\Gamma=\langle \gamma_1,\ldots,\gamma_n \mid r_l, \  l\in L \rangle
\end{equation}
with $L$ possible infinite. 
Let 
$$
\pi\co  \mathbb{F}_n \twoheadrightarrow \Gamma
$$
denote the natural surjection from the free group $\mathbb{F}_n$ on $n$ generators. Hence
the kernel of $\pi$ is the subgroup of $\mathbb{F}_n$ normally generated by the relations $r_l$, $l\in L$.  
Our first goal is  to describe the scheme of characters of $\Gamma$
in $\SL2$ from
the scheme of $\mathbb{F}_n$ and the presentation \eqref{eqn:fpGroup}.

The scheme of representations and characters in $\SL2$ are denoted respectively by $R(\Gamma,\SL2)$ and $X(\Gamma,\SL2)$.
The corresponding algebras of functions are the universal algebra of $\SL2$-representations
$$
A(\Gamma)=\CC [R(\Gamma,\SL2)]
$$
and the universal algebra of $\SL2$-characters
$$
B(\Gamma)=\CC [X(\Gamma,\SL2)]\cong A(\Gamma)^{\SL2}.
$$
We give the definitions in Section~\ref{sec:notation}.
The algebras $A(\Gamma)$ and $B(\Gamma)$ may be \emph{non-reduced}, i.e.~they may contain non-zero nilpotent elements 
(see the examples in 
Section~\ref{sec:examples}, for instance,
or the more general result of \cite{Kapovich-Millson} that provides the existence of
arbitrary
singularities). 
If instead of the scheme we consider the underlying variety,
then the algebra of functions of the variety is the quotient
$B(\Gamma)_{red} := B(\Gamma)/N(\Gamma)$, where $N(\Gamma)$ is the nilradical of $B(\Gamma)$.

For $\gamma\in \mathbb{F}_n$, let  $t_\gamma\in  B(\mathbb{F}_n)$ denote the the evaluation function of characters at $\gamma$.
\begin{Theorem} 
 \label{Thm:ideal}
 Let $\Gamma=\langle \gamma_1,\ldots,\gamma_n \mid r_l,\ l\in L \rangle$ be a finitely generated 
 group.  Then 
$$
 B(\Gamma)\cong B(\mathbb{F}_n)/\big(  t_{r_l}-2, \ t_{\gamma_i r_l}-t_{\gamma_i},\  t_{\gamma_j\gamma_k r_l}-t_{\gamma_j\gamma_k} \mid
l\in L,\ 1\leq i\leq n,\ 1 \leq j < k\leq n\big) 
 $$
\end{Theorem}

We compare this result to Theorem~3.2 in 
\cite{FicoMontesinos} 
by
Gonz\'alez-Acu\~ na and Montesinos-Amilibia, who  give a presentation
for the variety instead of the scheme.  
More precisely, they deal with the reduced algebra
and prove that  the reduction $B(\Gamma)_{red} $ of $B(\Gamma)$ is given by
$$
  B(\Gamma)_{red} \cong  B(\mathbb{F}_n)/ \operatorname{rad}\left(   t_{r_l}-2, \ t_{\gamma_i r_l}-t_{\gamma_i} \mid
      l\in L,\ 1\leq i\leq n \right) .
$$
Our motivation for Theorem~\ref{Thm:ideal} is that 
$ B(\Gamma)$ may be non-reduced and that, even when $B(\Gamma)$ is reduced, the 
radical in the ideal of Gonz\'alez-Acu\~ na and Montesinos-Amilibia
is needed,
see Example~\ref{ex:trefoil}).

In order to work with shorter words and to simplify computations,  we may 
write each  relation as a product of two words
\begin{equation}
\label{eqn:ruv}
r_l=u^{-1}_l v_l, \qquad \forall l\in L.
\end{equation}

\begin{Addendum} 
   Let $\Gamma=\langle \gamma_1,\ldots,\gamma_n \mid u_l =v_l, \ l\in L \rangle$ be a finitely generated 
 group.  Then
$$
 B(\Gamma)\cong B(\mathbb{F}_n)/
\big(  t_{v_l}-t_{u_l}, \ t_{\gamma_i v_l}-t_{\gamma_i u_l},\  t_{\gamma_j\gamma_k u_l}-t_{\gamma_j\gamma_k v_l} \mid
l\in L,\ 1\leq i\leq n,\ 1 \leq j < k\leq n\big).
 $$
  \end{Addendum}

The first step of the proof uses the isomorphism between $B(\Gamma)$ and the skein algebra of $\Gamma$, 
proved by Przytycki and Sikora in \cite{PrzytyckiSikora}, but that
follows also from a  result of Procesi \cite{Procesi87}. The proof uses trace 
identities of Vogt and Magnus to find the explicit presentation. 

In this paper we also define the scheme of \emph{non-simple} representations and characters
(some times called reducible representations, but we already use the word reduced
with another meaning). In Proposition~\ref{prop:abelianize} we prove that the scheme of 
non-simple characters is isomorphic to the scheme of characters of its abelianization.
We also describe the scheme of characters for abelian groups in Theorem~\ref{thm:abelian}.

In this article we are mainly interested in the field of complex numbers. 
Nevertheless, all results 
(in particular Theorem~\ref{Thm:ideal}) are 
valid over any algebraically closed field of characteristic $0$.

\begin{Remark}
 Recently G.~Miura and S.~Suzuki 
 have given a similar description of $B(\Gamma)$ when $\Gamma$
 has three generators 
\cite{MiuraSuzuki2021}. 
We discuss the differences between both results in 
Remark~\ref{rem:ideal}. 
\end{Remark}

The paper is organized as follows. Preliminaries on affine schemes are provided in
Section~\ref{sec:notation}, which can be skipped by readers familiar with schemes  
and used as a reference. Then in  Section~\ref{sec:schemereps} we discuss 
the scheme of representations, and the local properties of these are given in 
Section~\ref{section:local}. Section~\ref{sec:fpg} is devoted to the proof of Theorem~\ref{Thm:ideal}.
The schemes of non-simple and of abelian characters are studied respectively in 
Sections~\ref{sec:nonsimple} and \ref{sec:abelian}.
In Section~\ref{sec:examples} we provide examples of non-reduced character schemes of 
three-manifolds, and in the final section 
we compute the scheme of characters of the Borromean rings exterior.

\section{Preliminaries on affine algebraic schemes}
\label{sec:notation}

In this section we review the basic tools in affine algebraic schemes required. For more details see for instance 
\cite{EisenbudHarris,Milne,Mumford_red}. 

\subsection{Affine algebraic schemes} Let $A$ be a finitely generated, commutative $\CC$-algebra.
Recall  that $A$ is the quotient of a polynomial algebra by an ideal, i.e.~there exists $n\in\NN$ and an ideal $I\subset\CC[x_1,\ldots,x_n]$ such that $A\cong \CC[x_1,\ldots,x_n]/I$.
It follows from Zariski's lemma that there is a one-to-one correspondence between the maximal ideals $\mathfrak{m}\subset A$ and algebra morphisms $A\to\CC$. In what follows we let $\operatorname{Spm}(A)$ denote the maximal spectrum of $A$ i.e.
\[
\operatorname{Spm}(A) = \{ \mathfrak{m}\subset A\mid \mathfrak{m}\text{ is a maximal ideal.}\}
\]
We can think of elements of $A$ as complex-valued \emph{functions} on $\operatorname{Spm}(A)$ in the following way: if $f\in A$ and $\frak{m}\in \operatorname{Spm}(A)$, then $A/\frak{m}\cong \CC$ and $f(\frak{m}) := f \bmod \frak{m}$. Therefore, we have $f(\frak{m}) = 0$ if and only if $f\in \frak{m}\subset A$. Notice that we might have that $f(\frak{m}) =0$ for all $\frak{m}\in\operatorname{Spm}(A)$, but $f\neq 0$;
if his happens then $f\in A$ is nilpotent. Also each $\mathfrak{m}\in\operatorname{Spm}(A)$ determines, by evaluation, an algebra morphism $\phi_\mathfrak{m}\co A\to\CC$ given by
$\phi_\mathfrak{m}(f) = f(\mathfrak{m})$ for $f\in A$.
The two important properties are the following:
\begin{itemize}
\item The elements of $A$ separate points. Indeed, if $\frak{m},\frak{m}'\in\operatorname{Spm}(A)$ then
$\frak{m}\neq \frak{m}'$ implies that there exists $f\in \frak{m}$ and $f\not\in \frak{m}'$, and therefore $0=f(\frak{m})\neq f(\frak{m}')\neq0$. 
\item Every algebra morphism $\phi\co A \to \CC$ is the evaluation at a unique point i.e.
there exists a unique point $\frak{m} = \operatorname{Ker}(\phi) \in \operatorname{Spm}(A)$ such that for all $f\in A$ we have 
$\phi(f) = f(\frak{m})$.
\end{itemize}

For any subset $M$ of $A$ we define
\[
V(M) := \{\frak{m}\in\operatorname{Spm}(A)\mid  \frak{m}\supset M\}
= \{\frak{m}\in\operatorname{Spm}(A)\mid \forall f\in M,  f(\frak{m}) = 0\}\,.
\] 
It is clear that $V(M) = V(I)$ if $I$ is the ideal generated by $M$.
We will endow $\operatorname{Spm}(A)$ with the Zariski topology, where the closed sets are of the form
$V(I)$ for $I$ and ideal of $A$. Notice that $V(0)=\operatorname{Spm}(A)$ and $V(1) = \emptyset$. 
For more details see Appendix A in \cite{Milne}.

We will call a pair $X = (\operatorname{Spm}(A),A )$ an \emph{affine algebraic scheme}.
Given an affine algebraic scheme $X$, we define the \emph{coordinate algebra} of $X$ as  $\CC[X] := A$, and the \emph{underlying space} $|X| = \operatorname{Spm}(A)$.
In what follows it will be convenient to write $x\in|X|$ for an element in 
$|X|=\operatorname{Spm}(A)$. In this case the corresponding maximal ideal in $A$ will be denoted by~$\mathfrak{m}_x$.

\begin{Remark}
 It is worth noticing the word ``algebraic'' in the definition of \emph{affine algebraic scheme}. Affine schemes in general
 are constructed from the prime spectrum of a commutative ring with unity, though here we work with the maximal spectrum of a 
 finitely generated commutative $\mathbb C$-algebra.
 See \cite{Mumford_red} or \cite{EisenbudHarris}.
\end{Remark}

\begin{Example}
The $n$-dimensional affine space has the structure on an affine algebraic scheme
$\mathbb{A}^n : =(\CC^n,\CC[x_1,\ldots,x_n])$. A point $p=(p_1,\ldots,p_n)\in\CC^n$ corresponds to the maximal ideal $\mathfrak{m}_p=(x_1-p_1,\ldots,x_n-p_n)$, and each maximal ideal of
$\CC[x_1,\ldots,x_n]$ is of this form. The value of $f\in \CC[x_1,\ldots,x_n]$ at 
$\mathfrak{m}_p=(x_1-p_1,\ldots,x_n-p_n)$ is simply $f(p)$.
\end{Example}

\begin{Example}
 We consider the \emph{dual numbers} $A_2 := \CC[\epsilon]/(\epsilon^2)$. We have that
$\operatorname{Spm}(A_2) = \{\ast\}$ has only one point,  the maximal ideal $\ast:=(\epsilon)$.
The value of $f = a + b\,\epsilon\in A_2$ at the point $\ast$ is $f(\ast)=a$. Hence for $\epsilon \in A_2$ we have $\epsilon(\ast) = 0$, but $0\neq\epsilon$ and $\epsilon^2=0$ in $A_2$. The scheme $(\{\ast\},A_2)$ is called the \emph{double
point}, and similarly  $A_n := \CC[\epsilon]/(\epsilon^n)$ gives us a point of multiplicity $n$.
\end{Example}

A morphism between two affine algebraic schemes 
$X = (\operatorname{Spm}(A),A )$ and 
$Y= (\operatorname{Spm}(B),B )$
is a pair $(\alpha,\alpha^*)$  such that
$\alpha\co\operatorname{Spm}(A)\to \operatorname{Spm}(B)$ is a map, $\alpha^*\co B\to A$ is an algebra homomorphism satisfying 
\begin{equation}\label{eq:scheme_morph}
g(\alpha(\frak{m})) = \alpha^*(g)(\frak{m}) \quad\text{ for all $\frak{m}\in\operatorname{Spm}(A)$ and for all $g\in B$} 
\end{equation}
i.e.\ for all $g\in B$ the two maps $g\circ\alpha$ and $\alpha^*(g)$ are equal as maps on $|X|=\operatorname{Spm}(A)$. Notice that $\alpha$ is continuous in the Zariski-topology, for if $J\subset B$ is an ideal we have
\begin{align*}
\alpha^{-1} \big(V(J)\big) & =\{ \frak{m}\in \operatorname{Spm}(A) \mid \alpha(\frak{m}) \supset J\}\\
&= \{ \frak{m}\in \operatorname{Spm}(A) \mid \forall g \in J,\  \alpha^*(g)(\frak{m}) = g(\alpha(\frak{m})) =0\}\\
&= V\big(\alpha^*(J)\big)\,.
\end{align*}
In what follows we will write $\alpha\co X\to Y$ for a morphism between the schemes $X$ and $Y$. It is understood that $\alpha\co |X|\to|Y|$, and 
$\alpha^*\co\CC[Y]\to\CC[X]$ satisfy equation~\eqref{eq:scheme_morph} i.e.\ for all $g\in\CC[Y]$ and for all $x\in |X|$ we have $\alpha^*(g)(x) = g(\alpha(x))$.

\begin{Remark}
\begin{itemize}

\item An algebra homomorphism $\alpha^*\co B\to A$ determines a unique map
$\alpha\co\operatorname{Spm}(A)\to \operatorname{Spm}(B)$ in the following way:
for $\frak{m}\in\operatorname{Spm}(A)$ we consider the map $\phi_\frak{m}\co B\to\CC$ given by $\phi_\frak{m}(g) = \alpha^*(g)(\frak{m})$. This map is an algebra homomorphism and we can put  $\alpha(\frak{m}) := \operatorname{Ker}(\phi_\frak{m})$, and  we obtain 
$g(\alpha(\frak{m})) = \alpha^*(g)(\frak{m})$ for all $\frak{m}\in\operatorname{Spm}(A)$.
Notice that $\alpha(\mathfrak{m}) = (\alpha^*)^{-1}(\mathfrak{m})$ since 
$\operatorname{Ker}(\phi_\frak{m}) = (\alpha^*)^{-1}(\mathfrak{m})$.

\item The map $\alpha\co\operatorname{Spm}(A) \to \operatorname{Spm}(B)$ does not determine $\alpha^*$ uniquely. 
This only happens if $A$ has no nilpotent elements.
\end{itemize}
\end{Remark}

A scheme morphism 
$\alpha \co X\to Y$ is called a \emph{closed immersion} 
if the underlying continuous map $\alpha\co|X|\to|Y|$ is a homeomorphism between $|X|$ and a closed subset of $|Y|$, and the algebra homomorphism 
$\alpha^*\co \CC[Y]\to \CC[X]$ is surjective.

\subsection{Tangent space}
Let $X$ be an affine algebraic scheme.
For every $x\in|X|$, the quotient $\mathfrak{m}_x/\mathfrak{m}_x^2$ is a finite dimensional
$\CC$-vector space. By definition its dual space is the tangent space of $X$ at $x$ which will be denoted by
\[
T_x (X) := \big(\mathfrak{m}_x/\mathfrak{m}_x^2\big)^*\,.
\]

\begin{Lemma}\label{lem:tangent-scheme}
Let $\mathfrak{m}\in\operatorname{Spm}(A)$. 
Then there is a one-to-one correspondence between linear forms $\ell\in (\mathfrak{m}/\mathfrak{m}^2)^*$ and  scheme morphisms 
$(\alpha,\alpha^*) \co (\operatorname{Spm}(A_2),A_2)\to (\operatorname{Spm}(A),A)$ such that $\alpha(*) =\mathfrak{m}$.
\end{Lemma}
\begin{proof}
Recall that
$\operatorname{Spm}(A_2) = \{\ast\}$ has only one point  which corresponds to the maximal ideal $(\epsilon)$.
We consider the map
$\alpha\co\operatorname{Spm}(A_2)\to \operatorname{Spm}(A)$ given by
$\alpha(\ast) = \mathfrak{m}$. 
In order to get a morphism of schemes, we are looking for an algebra homomorphism 
$\alpha^*\co A\to A_2$ which satisfies equation~\eqref{eq:scheme_morph}. That is 
\[
\alpha^*\co A\to A_2 \text{ given by } \alpha^*(f) = \alpha_0^*(f) + \epsilon\,\alpha_1^*(f)
\]
where $\alpha_i^*\co A\to\CC$ are $\CC$-linear maps, such that
\[
\alpha^*(f)(\ast) = \alpha_0^*(f)=f(\alpha(\ast)) = f(\mathfrak{m})\,,
\] 
and 
\begin{equation}\label{eq:tangent}
\forall\ f_1,f_2\in A, \quad \alpha_1^*(f_1 f_2) = f_1(\mathfrak{m}) \alpha_1^*(f_2) + \alpha_1^*(f_1) f_2(\mathfrak{m})\,.
\end{equation}
Notice that equation~\eqref{eq:tangent} implies that $\alpha_1^*(1) =0$, and hence $\alpha_1^*|_\CC \equiv 0$.

Let $\ell\in  (\mathfrak{m}/\mathfrak{m}^2)^*$ be given. We define a linear map 
$\alpha_1^*\co A\to \CC$ satisfying equation \eqref{eq:tangent} by putting $\alpha_1^*(f) := \ell(f-f(\mathfrak{m}))$. Indeed,
\begin{align*}
\alpha_1^*(f_1 f_2) &= \ell(f_1f_2-f_1(\mathfrak{m})f_2(\mathfrak{m})) \\
&= \ell( (f_1- f_1(\mathfrak{m})) (f_2-f_2(\mathfrak{m}) )) + f_2(\mathfrak{m})\ell(f_1 - f_1(\mathfrak{m}) ) + f_1(\mathfrak{m}) \ell(f_2 - f_2(\mathfrak{m}) ) \,,
\end{align*}
and equation~\eqref{eq:tangent} follows since $(f_1- f_1(\mathfrak{m})) (f_2-f_2(\mathfrak{m})) \in \mathfrak{m}^2$ and
$\ell|_{\frak{m}^2} \equiv 0$.

Let 
$(\alpha,\alpha^*) \co (\operatorname{Spm}(A_2),A_2)\to (\operatorname{Spm}(A),A)$ 
be a  scheme morphisms such that $\alpha(*) =\mathfrak{m}$ i.e.\ for all $f\in A$
$\alpha^*(f) = f(\mathfrak{m}) + \epsilon\, \alpha_1^*(f)$ where $\alpha_1^*\co A\to\CC$ is linear and satisfies equation~\eqref{eq:tangent}. The restriction $\alpha_1^*|_\mathfrak{m}$ is linear, and by equation~\eqref{eq:tangent} $\alpha_1^*$ vanishes on $\mathfrak{m}^2$.
Hence $\alpha_1^*$ defines a linear form on $\mathfrak{m}/\mathfrak{m}^2$.
\end{proof}

\subsection{Closed subschemes} Let $I\subset A$ be an ideal, and we let $\pi\co A\to A/I$ denote the projection.
We obtain the affine algebraic scheme $(\operatorname{Spm}(A/I), A/I)$. Notice that there is a one-to-one correspondence between the ideals in  $A/I$ and the ideals in $A$ which contain~$I$.
Hence, there is a natural scheme morphism 
\[
(\alpha, \alpha^*)\co (\operatorname{Spm}(A/I), A/I )\to (\operatorname{Spm}(A), A )\quad\text{given by  
$\alpha (\overline{\mathfrak{m}}) = \pi^{-1}(\overline{\mathfrak{m}})$ and $\alpha^* =\pi$.}
\]
We have $\operatorname{Im}(\alpha)= V(I)$, and $\alpha\co\operatorname{Spm}(A/I)\to V(I)$ a homeomorphism since it is closed; that is
for every ideal $\bar{J}\subset A/I$ we have
\[
\alpha ( V(\bar{J}) ) = V(I)\cap V\big( \pi^{-1}( \bar{J} ) \big)\,.
\]

We endow the closed set $V(I)$ with the ring $A/I$, and we call the pair
$(V(I), A/I)$ a \emph{closed subscheme} of $(\operatorname{Spm}(A), A)$.
In
contrast to classical algebraic geometry, the closed subschemes
of $(\operatorname{Spm}(A), A )$ and the ideals of $A$ are in one-to-one correspondence.
It follows that every affine algebraic scheme is isomorphic to a closed subscheme of some affine space~$\mathbb{A}^n$.

For two closed subschemes $(V(I_k), A/I_k)$, $k = 1, 2$, the subscheme
$(V(I_1\cap I_2), A/(I_1\cap I_2))$
 is called their \emph{union} and $(V(I_1+ I_2), A/(I_1+ I_2))$ is called
their \emph{intersection}. Since
and
these are actually set-theoretic unions and intersections of the underlying spaces. See Section~I.2.1 in \cite{EisenbudHarris}. 

We call an affine algebraic scheme $X$ \emph{reduced} if $\CC[X]$ is reduced.
If $N\subset \CC[X]$ denotes the nilradical of $\CC[X]$, then the algebra 
$\CC[X]_{red} := \CC[X]/N$ is reduced. Moreover, $N$ is the intersection of all maximal ideals of $A$. 
Hence $V(N)=V(0) = |X|$, and it follows that 
$X_{red}:= (|X|, \CC[X]_{red})$ is a reduced, closed subscheme of $X$.

\subsection{Local to global properties} 
For a finitely generated algebra $A$ we let  $N(A)$ denote its nilradical. 
For an ideal $I\subset A$ we let $\operatorname{Ann}(I):=\{a\in A\mid a I = 0\}$ denote the \emph{annihilator} of $I$.
The point of $\operatorname{Spm}(A)$ represented by a maximal ideal $\mathfrak{m}$ is called \emph{reduced} if the local ring $A_\mathfrak{m}$ is reduced. The non-reduced points of
$\operatorname{Spm}(A)$ form a Zariski closed subset (which might be empty). This follows from the the following lemma.

\begin{Lemma}
\label{lem:non-reduced_closed} Let $A$ be a finitely generated, commutative $\CC$-algebra $A$, $I\subset A$ an ideal, and $\frak{m}\subset A$ a maximal ideal.  Then
\begin{enumerate}
\item $N(A_\frak{m}) = N(A)_\frak{m}$;
\item $ V\big(\operatorname{Ann}(I)\big) =\{ \frak{m}\in \operatorname{Spm}(A)\mid I_\frak{m} \neq 0\}$;
\item The set of non-reduced points in  $\operatorname{Spm}(A)$ is Zariski-closed.
\end{enumerate}
\end{Lemma}
\begin{proof}
\begin{enumerate}
\item If $\frac{a}{s}\in N(A_\mathfrak{m})$  then there exists $n\in\mathbb{N}$ such that $\big(\frac{a}{s}\big)^n = \frac{a^n}{s^n} =\frac01$. Hence there exists $t\in A\setminus\mathfrak{m}$ such that $a^nt =0$ in $A$. Hence $(ta)^n=0$ in $A$, and
$\frac{a}{s}= \frac{at}{st}\in N(A)_\mathfrak{m}$. The other inclusion is obvious.

\item We have 
\[
\mathfrak{m} \not\in V\big(\operatorname{Ann}(I)\big) \Leftrightarrow
 \mathfrak{m}\not\supset \operatorname{Ann}(I) \Leftrightarrow \exists\, s\in A\smallsetminus \mathfrak m, \, \forall\, b\in I \ s\, b =0
\]
 and
\[
I_\mathfrak{m} \neq 0 \Leftrightarrow \forall \, b \in I,\, \exists s\in A\smallsetminus\mathfrak{m} \ s\,b = 0
\] 
Clearly the first statement implies the second. But also the second implies the first since $A$ is noetherian.  For if $I=(b_1,\ldots,b_k)$ and $s_i \in A\smallsetminus\mathfrak{m}$ such that $s_ib_i = 0$ then for the product  $s=s_1\cdots s_k$ we have $sb=0$ for all $b\in I$.

\item We have $N(A_\frak{m})  =  N(A)_\frak{m} \neq 0$ if and only $\frak{m}\in V\big(\operatorname{Ann}(N(A))\big)$. Hence the non-reduced points $\frak{m}\in\operatorname{Spm}(A)$ form the Zariski-closed set $V\big(\operatorname{Ann}(N(A))\big)$.\qedhere
\end{enumerate}
\end{proof}

The first local to global property is the following (see  \cite[Corollary~5.19]{MilneCA}):
\begin{Lemma} \label{lem:loc_to_glob_red}
Let $A$ be a ring. Then $A$ is reduced if and only if for every maximal ideal $\mathfrak{m}\subset A$ the local ring
$A_\mathfrak{m}$ is reduced.
\end{Lemma}

Let $A$ be a subring of $B$. An element $b\in B$ is called \emph{integral} over $A$ if it is a root of a monic polynomial with coefficients in $A$. An integral domain $A$ is called \emph{integrally closed} or  \emph{normal} if it is integrally closed in its field of fractions $F$ i.e.\
$\alpha \in F$, and $\alpha$ integral over $A$ implies 
$\alpha\in A$. There is an other local to global property
(see \cite[Prop.~6.16]{MilneCA}:
\begin{Lemma} \label{lem:loc_to_glob_int}
Let $A$ be an integral domain. Then 
$A$ is normal if and only if for every maximal ideal 
$\mathfrak m\subset A$ the localization 
$A_{\mathfrak m}$ is normal. 
\end{Lemma}

\subsubsection{Tangent cone}
Let $B$ be a noetherian, local ring with maximal ideal $\mathfrak{m}$. In this situation we have by Krull's intersection theorem that 
\begin{equation}\label{eq:Krull}
 \displaystyle\bigcap_{n\in\mathbb{N}} \mathfrak{m}^n =\{0\},
\end{equation}
and we will define the associated \emph{associated graded ring of $B$} as 
\[ \mathrm{gr}(B)
= B / \mathfrak{m} \oplus  \mathfrak{m}/ \mathfrak{m}^2 \oplus \mathfrak{m}^2/ \mathfrak{m}^3 \oplus \cdots \,.
\]
The multiplication of two \emph{homogeneous} elements 
$\bar a \in \mathfrak{m}^{n}/\mathfrak{m}^{n+1}$ and
$\bar b \in \mathfrak{m}^{k}/\mathfrak{m}^{k+1}$ is defined as follows: 
\[
 \bar {a} \cdot \bar {b} := a\,b \bmod \mathfrak{m}^{n+k+1}
\]
where $a$ and $b$ are elements of $B$ representing 
$\bar a$ and $\bar b$ respectively.

For any finitely generated, commutative $\CC$-algebra $A$ the \emph{tangent cone} of the schema 
$X=(\operatorname{Spm}(A), A)$ at $x\in X$ is the schema 
\[
\big(\operatorname{Spm}(\mathrm{gr}(A_\mathfrak{m})),\mathrm{gr}(A_\mathfrak{m})\big)\text{ where }\mathfrak m = \mathfrak{m}_x\,.
\]

A local noetherian ring $B$ with maximal ideal $\mathfrak m$ inherits some properties of its associated graded algebra $\operatorname{gr}(B)$.
For example, if $\operatorname{gr}(B)$ is reduced, normal or regular, then so is the local ring $B$.
More precisely, it follows from \eqref{eq:Krull} that for $a\in B$, $a\neq0$,  there is a minimal $n$ such that 
$a\in \mathfrak{m}^n$, $a\not\in\mathfrak{m}^{n+1}$, 
and we define the \emph{initial term}  
$\operatorname{in}(a) \in \mathrm{gr}(B)$ of $a$ as
\[
\operatorname{in}(a) : = a \bmod \mathfrak{m}^{n+1}\in \mathfrak{m}^{n}/\mathfrak{m}^{n+1}\subset
\mathrm{gr}(B)\,.
\]
It is easy to see that
$
\operatorname{in} ( a b ) = \operatorname{in} ( a )\cdot \operatorname{in} ( b )
$.
This implies 
\begin{Proposition}\label{prop:reduced_normal}
Let $B$ be a local noetherian ring with maximal ideal $\mathfrak m$. 
\begin{enumerate}
\item If the associated graded algebra $\operatorname{gr}(B)$ is reduced, then $B$ is reduced.
\item If the associated graded algebra $\operatorname{gr}(B)$ is normal, then $B$ is normal.
\end{enumerate}
\end{Proposition}
See \cite[Section 17]{Matsumura} and 
\cite[Chapter 5]{Eisenbud} for more details.

%
%
%
%
%

\section{The scheme of representations and characters}
\label{sec:schemereps}
The general reference for this section is Lubotzky's and Magid's book \cite{LubotzkyMagid}.
In what follows we let $\Gamma$ denote a finitely generated group with presentation~\eqref{eqn:fpGroup}.

\subsection{The universal algebra $A(\Gamma)$.}\label{sec:universal_algebra}
The $\CC$-algebra $A(\Gamma)$ comes with a representation $\rho_u\co  \Gamma \to \SLL_2(A(\Gamma))$, and the pair $(A(\Gamma),\rho_u)$ has the following universal property: for every commutative $\CC$-algebra $A$ and every representation $\rho\co\Gamma\to\SLL_2(A)$ there exists a unique algebra morphism $f\co  A(\Gamma)\to A$ such that $\rho = f_*\circ \rho_u$ where $f_*\co \SLL_2(A(\Gamma))\to\SLL_2(A)$ is the induced map. 

The algebra $A(\Gamma)$ can be constructed from the 
presentation of $\Gamma$ in~\eqref{eqn:fpGroup}, see \cite[Prop.~1.2]{LubotzkyMagid}. 
In more detail, the algebra $A(\Gamma)$ is a quotient of the finitely generated polynomial algebra, and the construction goes at follows:
put 
\[
X^{(k)} = \left( 
\begin{smallmatrix} x^{(k)}_{11} & x^{(k)}_{12}\\ x^{(k)}_{21} & x^{(k)}_{22} \end{smallmatrix}
\right) \text{ for  $1\leq k\leq n$.}
\]
and consider the polynomial ring on $ 4 k$ variables 
\( 
 \CC[x_{ij}^{(k)}]\),    where $1\leq i,j \leq 2$ and $1\leq k\leq n$.
Then $A(\Gamma)$ is the quotient of $\CC[x_{ij}^{(k)}]$ by the ideal $J$ generated by
\[
\det(X^{(k)})-1,\text{ $1\leq k\leq n$, and } 
(m_l)_{ij}-\delta_{ij}, \ 1\leq i,j\leq2,\ l\in L 
\]
where $m_l = r_l(X^{(1)},\ldots,X^{(n)})$ is a matrix in $\operatorname{GL}( \CC[x_{ij}^{(k)}]  )$,
$(m_l)_{ij}$ are the entries of $m_l$, and $\delta_{ij}$ is the Kronecker delta.

\subsection{The representation scheme and the representation variety}
For a finitely generated group $\Gamma$ the
$\SL2$-\emph{scheme of representations}  $ R(\Gamma, \SL2)$ is defined as pair 
\[
R(\Gamma, \SL2) := \big(\operatorname{Spm}(A(\Gamma)),A(\Gamma)\big)
\]
where $A(\Gamma)$ is the universal algebra of $\Gamma$ (see Section~\ref{sec:universal_algebra}).

A single representation $\rho\co \Gamma\to\SL2$ corresponds to $\CC$-algebra morphism
$A(\Gamma)\to \CC$ i.e.\ to a maximal ideal $\mathfrak m_\rho \subset A(\Gamma)$, and each maximal ideal $\mathfrak m\subset A(\Gamma)$ determines a representation
$\rho_\mathfrak{m}\co \Gamma\to\SL2$.
Hence maximal ideals in $A(\Gamma)$ correspond exactly to representations
$\Gamma\to\SL2$, and we will use this identification in what follows. 

Regular functions on $R(\Gamma, \SL2) $ are exactly the elements of $A(\Gamma)$. More precisely, if $f\in A(\Gamma)$ and $\rho\in R(\Gamma, \SL2) $, then 
\[
f(\rho) = f\bmod \mathfrak m_\rho \in A(\Gamma)/\mathfrak m_\rho\cong\CC\,.
\]

If $N$ is the nil-radical of $A(\Gamma)$, then the  reduced algebra 
$A(\Gamma)_\mathrm{red} = A(\Gamma)/N$ is isomorphic to the 
coordinate ring of the representation variety. This justifies the notation
$R(\Gamma,\SL2)_{\mathrm{red}}$ for the representation variety. The surjection
$A(\Gamma) \twoheadrightarrow A(\Gamma)_\mathrm{red}$ induces a \emph{closed immersion}
$R(\Gamma,\SL2)_{\mathrm{red}}\to R(\Gamma,\SL2)$, the underlying continuous map is a homeomorphism in the Zariski topology. 
In particular the scheme $R(\Gamma,\SL2)$ and the variety $R(\Gamma,\SL2)_{\mathrm{red}}$ have the same dimension.
The scheme $R(\Gamma,\SL2)$ might support more regular functions than 
$R(\Gamma,\SL2)_\mathrm{red}$. More precisely, different regular functions on the scheme can have the same values on maximal ideals.

If the nil-radical of $A(\Gamma)$ is trivial then we call the algebra $A(\Gamma)$ and the representation scheme 
$R(\Gamma,\SL2)$ \emph{reduced}.

\subsection{The universal algebra $B(\Gamma)$,  the character scheme, and trace functions}
The group $\SL2$ acts on the $2\times 2$ matrices by conjugation. 
This induces an action of $\SL2$ on the algebra $A(\Gamma)$, and the algebra 
$B(\Gamma)= A(\Gamma)^{\SL2}\subset A(\Gamma)$ of invariants is called the \emph{universal algebra of $\SL2$-characters}. The algebra $B(\Gamma)$ is also a quotient of a finitely generated polynomial algebra, and the affine GIT quotient is naturally defined as the  
\emph{scheme of characters} 
\[
X(\Gamma, \SL2) := (\operatorname{Spm}(B(\Gamma)), B(\Gamma))\,,
\] see \cite{LubotzkyMagid}. 
A single character $\chi\co \Gamma\to\CC$ corresponds to $\CC$-algebra morphism
$B(\Gamma)\to \CC$ i.e.\ to a maximal ideal in $B(\Gamma)$. 
Also the algebra $B(\Gamma)$ might be non-reduced.

\begin{Remark}
Notice that both $ R(\Gamma, \SL2)$ and $X(\Gamma, \SL2)$ can be reducible.
\end{Remark}

Given an element $\gamma\in\Gamma$ there exists a word
$w_\gamma(\gamma_1,,\ldots,\gamma_n)\in \mathbb F_n$ which represents $\gamma$.
The matrix $w_\gamma(X^{(1)},\ldots,X^{(n)})\in \operatorname{GL}_2(\CC[x_{ij}^{(k)}])$,
and 
$$t_\gamma := \tr \big(w_\gamma(X^{(1)},\ldots,X^{(n)})\big)$$ 
represents an element in $B(\Gamma)$. We call $t_\gamma$ the \emph{trace function} associated to 
$\gamma\in\Gamma$. 

The \emph{first fundamental theorem} of invariant theory asserts that the trace functions $t_\gamma$, $\gamma\in \Gamma$,
generate the algebra $B(\Gamma)$ (see \cite{LubotzkyMagid} for instance). Besides finding an explicit finite set of generators,
we aim to describe the relations satisfied by the $t_\gamma$. 
For the moment we just state some basic relations that follow from properties of the trace.
Namely for all 
$a,b\in\Gamma $ and the neutral element $e\in\Gamma$ we have:
\begin{equation}\label{eq:trace_rel}
t_e = 2,\quad t_a = t_{a^{-1}},\quad t_{ab} = t_{ba},\quad\text{ and }\quad
t_{ab} + t_{ab^{-1}} = t_{a}t_{b}.
\end{equation}
The last equality follows from the  Cayley--Hamilton theorem.


\subsection{Free groups} 
The general reference for this section is Goldman's Handbook article \cite{Goldman09} and the references therein. 

Let $\mathbb F_n$ be a free group of rank $n$.
It is a classical result that the algebra $B(\mathbb{F}_n)$ is reduced. By works of Cartier \cite{Cartier}, the coordinate ring $\CC[G]$ of an algebraic group 
$G$ in characteristic zero is reduced (see Section~3.h of \cite{Milne} for a modern approach). But for $\SL2$ a more elementary argument is sufficient.
\begin{Theorem}
 The algebra $B(\mathbb{F}_n)$ is reduced.
\end{Theorem}
\begin{proof}
The polynomial 
$x_1x_4-x_2x_3-1\in \CC[x_1,x_2,x_3,x_4]$ is irreducible, and an irreducible element in a factorial ring is prime. Therefore the coordinate algebra 
$$\CC[\SL2] = \CC[x_1,x_2,x_3,x_4]/(x_1x_4-x_2x_3-1)$$ is a domain, and hence reduced.
Moreover,  we obtain that
$$\displaystyle A(\mathbb{F}_n)\cong \bigotimes_{k=1}^n \CC[\SL2]$$ is reduced  
\cite[V, \S15, Theorem~3]{Bourbaki:algII}. It follows that 
$B(\mathbb{F}_n)= A(\mathbb{F}_n)^{\SL2}\subset A(\mathbb{F}_n)$ is also reduced.
\end{proof}
An explicit presentation of $B(\mathbb{F}_n)$ is given by  
Ashley, Burelle, and Lawton  in~\cite{Ashley2018}.

For a free group  $\mathbb F_n = \langle \gamma_1,\ldots,\gamma_n\mid \emptyset \rangle $ 
of rank $n$ the algebra
$B(\mathbb F_n)$ is generated by the following $n(n^2+5)/6$ elements
\[
t_{\gamma_i},\ 1\leq i\leq n,\quad 
t_{\gamma_i\gamma_j},\ 1\leq i<j\leq n,\quad\text{and}\quad 
t_{\gamma_i\gamma_j\gamma_k},\ 1\leq i<j<k\leq n
\]
(see \cite{FicoMontesinos}).
In the following examples we give the explicit presentation  $B(\mathbb{F}_n)$ in the rank two and three.

\begin{Example} The rank two case goes back to
  Fricke and Klein; they  proved that $B(\mathbb{F}_2)$ is isomorphic the polynomial ring 
$\CC[t_{\gamma_1},t_{\gamma_2},t_{\gamma_1\gamma_2}]$ in three variables (see Section 2.2 in~\cite{Goldman09}).
\end{Example}

\begin{Example} In the case of rank three $B(\mathbb{F}_3)$  is not isomorphic to a polynomial algebra.
More precisely,  for a free group on three generators $\mathbb F_ 3=\langle a,b,c\mid\emptyset\rangle$, 
the coordinate ring of $X( \mathbb F_3,\SL2)$ 
has been computed for instance  in \cite{FicoMontesinos, Magnus}:
\begin{align}
B( \mathbb F_3 ) & = 
\CC[t_a,t_b,t_c,t_{ab},t_{ac},t_{bc},t_{abc}]/(F) \label{eq:CoRi}
\intertext{where}
F &= F(t_a,t_b,t_c,t_{ab},t_{ac},t_{bc},t_{abc}) = t_{abc}^2 - p\, t_{abc} + q \label{eq:F}
\intertext{with $p,q\in \CC[t_a,t_b,t_c,t_{ab},t_{ac},t_{bc}]$ given by}
p &= t_at_{bc} + t_bt_{ac} + t_ct_{ab}- t_at_bt_c \label{eq:p}
\intertext{and}
q  &= t_a^2 + t_b^2 + t_c^2 + t_{ab}^2 + t_{ac}^2 + t_{bc}^2 +
t_{ab}t_{ac}t_{bc} 
- t_at_bt_{ab} - t_at_ct_{ac} - t_bt_ct_{bc} -4\,.\label{eq:q}
\end{align}
The trace functions $t_{abc}$, $t_{acb}$ are solutions of the quadratic equation
over $\CC[t_a,t_b,t_c,t_{ab},t_{ac},t_{bc}]$
\begin{equation}\label{eq:quad_eq}
 X^2 - p\, X +q =0\,.
\end{equation}
We have 
\begin{equation}\label{eq:p_and_q}
t_{abc}+ t_{acb} = p,\quad t_{abc}t_{acb} =q\quad\text{ and }\quad (t_{abc}-t_{acb})^2 = \Delta
\end{equation}
where
\begin{equation}\label{eq:disc}
\Delta = p^2 -4q\,.
\end{equation}
(see Section 5.1 in \cite{Goldman09}).
\end{Example}

\section{Local properties}
\label{section:local}

In this section we keep on reviewing properties of the scheme of representations
and the scheme of characters, now focusing on local properties.
For instance, 
in order to check that the scheme is reduced, it is sufficient to prove that it is locally reduced.

\subsection{Zariski tangent space}
For a representation $\rho\co \Gamma\to \SL2$, $Z^1(\Gamma,\Ad\rho)$
denotes the space of 1-cocycles or crossed morphisms twisted by 
$\Ad\rho$, namely linear maps
$$
d\co \Gamma\to \sl2 
$$
that satisfy 
$$
d(\gamma_1\gamma_2)=d(\gamma_1)+\Ad_{\rho(\gamma_1)} d(\gamma_2),\qquad \forall \gamma_1,\gamma_2\in \Gamma.
$$
Weil's construction \cite{Weil}  gives the following theorem (see also \cite{LubotzkyMagid}):
\begin{Proposition}
The Zariski tangent space to the scheme  $R(\Gamma,\SL2)$ at
$\rho$ is naturally isomorphic to $Z^1(\Gamma,\Ad\rho)$.
\end{Proposition}
\begin{proof} By Lemma~\ref{lem:tangent-scheme}, we have that a tangent vector of
$R(\Gamma,\SL2)$ at $\rho$ corresponds to a scheme morphism
\[
(\alpha,\alpha^*) \co (\{*\},A_2)\to \big(\operatorname{Spm}\big(A(\Gamma)\big),A(\Gamma)\big)
\]
such that $\alpha(*) = \rho$, and $\alpha^*\co A(\Gamma)\to A_2$ is an algebra homomorphism. In turn,
this corresponds to a representation $\rho_\alpha\co\Gamma\to\SLL_2(A_2)$ such that
for all $\gamma\in \Gamma$
\[
\rho_\alpha(\gamma) = \big(\id + \epsilon\, d(\gamma)\big) \rho(\gamma)\ \text{such that $d\co\Gamma\to\sl2$.}
\]
Notice that $\SLL_2(A_2) = \{ (\id +\epsilon X) A\mid A\in\SL2\text{ and } \tr(X)=0\}$.

Now, it is easy to check that $\rho_\alpha$ is a homomorphism if and only if 
$d\co \Gamma\to\sl2$ is a cocycle.
\end{proof}

The cohomology  of $\Gamma$ with coefficients in the
$\Gamma$-module $\sl2$ twisted by $\Ad\rho$ is isomorphic to
$$
H^1(\Gamma,\sl2)\cong Z^1(\Gamma,\Ad\rho)/B^1(\Gamma,\Ad\rho),
$$
where $B^1(\Gamma,\Ad\rho)$ denotes the subspace of inner crossed morphisms, namely crossed morphisms 
$d\co \Gamma\to \sl2$ for which there exists $a\in \sl2$ so that 
$$
d(\gamma)=a-\Ad_{\rho(\gamma)}(a),
\qquad \forall\gamma\in\Gamma.
$$

\begin{Theorem}
\label{thm:H1ZT}
If $\rho$ is simple, then 
the Zariski tangent space to $X(\Gamma, \SL2)$ at $\chi_\rho$
is naturally isomorphic to $ H^1(\Gamma,  \Ad\rho)$.
\end{Theorem}

This is \cite[Thm~2.13]{LubotzkyMagid}.

Recall from the previous section that the variety of representations
$R(\Gamma,\SL2)_\mathrm{red} $ is the union of affine varieties. 
We denote by 
$\dim_\rho R(\Gamma,\SL2)_\mathrm{red}$
the maximal dimension of all components
of $R(\Gamma,\SL2)_\mathrm{red}$ containing $\rho$.

\begin{Lemma}
\label{lemma:schemesmooth}
For any representation $\rho\in R(\Gamma,\SL2)$, we have
\[
\dim_\rho R(\Gamma,\SL2)_\mathrm{red} \leq \dim Z^1(\Gamma;\Ad\rho),
\]
with equality if, and only if,  $\rho$ is reduced and $\rho$ is 
a smooth point of the representation variety.
\end{Lemma}

\begin{proof}
 The proof is an easy consequence of the following inequalities:
\begin{multline*}
\dim_\rho R(\Gamma,\SL2)_{\mathrm{red}}\leq \dim T_\rho \left(R(\Gamma,\SL2)_{\mathrm{red}} \right) \\ \leq \dim T_\rho R(\Gamma,\SL2) = \dim Z^1(\Gamma,\Ad\rho),
\end{multline*}
where $T_\rho$ denotes the Zariski tangent space at $\rho$ of both 
the variety or the scheme. 
\end{proof}

\begin{Definition}[See \cite{Sikora}]
  We call $\rho\in R(\Gamma,\SL2)$ \emph{scheme smooth} if
\[
\dim_\rho R(\Gamma,\SL2)_\mathrm{red} = \dim Z^1(\Gamma;\Ad\rho),
\]
\end{Definition}

\begin{Example}
\label{ex:finite}
If $\Gamma$ is a finite group, then every  $\rho\in R(\Gamma,\SL2)$ is scheme smooth, because 
$H^1(\Gamma,  \Ad\rho)=0$ and the following lemma.
\end{Example}

\begin{Lemma}
 If $H^1(\Gamma,  \Ad\rho)=0$, then $\rho$ is  scheme smooth.
\end{Lemma}

\begin{proof}
We use that every $d\in Z^1(\Gamma,\Ad\rho)$ is inner, namely there is $a\in\sl2$
such that $d(\gamma)=a-\Ad_\rho(\gamma) a$ for every $\gamma\in\Gamma$. It can be checked that
$d$ is a vector tangent to the orbit by conjugation, to the path of conjugation  by $\exp(t a)$.
This yields $\dim_\rho R(\Gamma,\SL2)_\mathrm{red} = \dim Z^1(\Gamma,\Ad\rho)$.
\end{proof}

In what follows we call a representation 
$\rho\co\Gamma\to\SL2$ \emph{reduced} and \emph{normal} if the local ring  
 $A(\Gamma)_{\mathfrak m_\rho}$ is reduced and normal, respectively.
 Here $ \mathfrak m_\rho$ denotes the maximal ideal associated to $\rho$.

\medskip

The tangent cone $TC_\rho R(\Gamma,\SL2)$ is the spectrum of the graded $\CC$-algebra associated to the local ring $A(\Gamma)_{\mathfrak m_\rho}$. 
Moreover, the tangent space $T_\rho R(\Gamma,\SL2)$ is the smallest affine subspace which contains the tangent cone  (see \cite[III.\S3]{Mumford_red}). 
From Proposition~\ref{prop:reduced_normal} we obtain:
\begin{Lemma}\label{lem:reduced}
Let $\rho\co\Gamma\to \SL2$ be a representation.

Then $\rho$ is reduced and normal if the tangent cone 
$TC_\rho R(\Gamma,\SL2)$ is reduced and normal, respectively.
\end{Lemma}

\section{Computing the character scheme of a finitely generated group}
 \label{sec:fpg} 
In
this section we prove Theorem~\ref{Thm:ideal}, and
for that purpose we recall 
the definition of the skein algebra. For a finitely generated group we let  $\CC\Gamma$ denote the group ring of $\Gamma$. The tensor algebra 
over $\CC\Gamma$ is denoted by $\mathrm{T}(\CC\Gamma)$.
 
\begin{Definition}
 The \emph{skein algebra} of $\Gamma$ is
 $$\mathcal S _\Gamma =  \mathrm{T}(\CC\Gamma)/\left(   e-2, \ \alpha\otimes\beta-\beta\otimes\alpha,\ \alpha\otimes\beta-\alpha\beta-\alpha\beta^{-1}\mid \alpha,\beta\in\Gamma   \right)\,.
$$
\end{Definition}
The definition of the skein algebra is motivated by the
trace functions identities in equation~\eqref{eq:trace_rel}.
Based on work of Brumfiel and Hilden \cite{BrumfielHilden},
Przytycki and Sikora proved in \cite{PrzytyckiSikora} the following:
\begin{Theorem}[\cite{PrzytyckiSikora}] There is an isomorphism of $\CC$-algebras
$\Phi\co  \mathcal S_\Gamma  \xrightarrow{\cong}  B(\Gamma)$ determined by
$\Phi([\gamma]) = t_\gamma$. 
\end{Theorem}
As explained by March\'e in \cite{Marche_skein}, 
the theorem also follows from a more general result of Procesi \cite{Procesi87}.

Recall  that we have a presentation
$
\Gamma=\langle \gamma_1,\ldots,\gamma_n \mid r_l,\ l\in L \rangle
$ as in \eqref{eqn:fpGroup}.
 The natural projection  $$\pi\co  \mathbb{F}_n\to\Gamma$$ 
induces a surjection $\pi_*\co  \mathcal{S}_{\mathbb F_n} \twoheadrightarrow \mathcal{S}_\Gamma$
that factors to an epimorphism
\[
\bar \pi\co  \mathcal{S}_{\mathbb F_n}/\mathcal K \twoheadrightarrow \mathcal{S}_\Gamma ,
\]
where $\mathcal K\subset \mathcal{S}_{\mathbb F_n} \cong B(\mathbb F_n)$ is the   ideal
\[
\mathcal K = \left( 
	  [\alpha]-[\alpha']\mid \alpha\in\mathbb{F}_n,\  \pi(\alpha)=\pi(\alpha')
\right)\,.
\]
\begin{Lemma}  
The map $\bar \pi\co  \mathcal{S}_{\mathbb F_n}/\mathcal K \twoheadrightarrow \mathcal{S}_\Gamma$ is 
an isomorphism of $\CC$-algebras.
\end{Lemma}

\begin{proof}
To construct the inverse, start with   a set-theoretic section $s\co  \Gamma\to\mathbb{F}_n$ of the projection $\pi\co  \mathbb{F}_n\to\Gamma$,
extend it linearly to $\mathrm{ T }(\CC\Gamma)\to \mathrm{ T }(\CC \mathbb{F}_n) $ and compose the extension with the projection to
$\mathcal{S}_{\mathbb F_n}/\mathcal K$.

By construction of $\mathcal K$, we have for all $\alpha,\beta\in \Gamma$,
$s(\alpha\beta)- 	 s(\alpha) s(\beta) \in\mathcal K $, therefore $s$ induces 
a morphism of algebras   
$\bar s\co  \mathcal{S}_\Gamma\to  \mathcal{S}_{\mathbb F_n}/\mathcal K$,
 that satisfies $[\alpha]=\bar s(\bar \pi([\alpha]))$ for all $ \alpha\in \mathbb F_n$ and
$[\beta]=\bar\pi(\bar s([\beta]))$ for all $\beta\in \Gamma$. Since  classes of the elements in the group  span linearly the skein algebra, 
$\bar s$ and $\bar \pi$ are inverses of one another.
\end{proof}

\begin{Corollary}\label{Coro:BF/} Let $
I= \left(  t_\alpha-t_{\beta}\mid \alpha,\beta\in\mathbb{F}_n,\ \pi(\alpha)=\pi(\beta)  \right)
$. Then 
\[
 B(\Gamma)\cong B(\mathbb{F}_n)/ I.
\]
 
\end{Corollary}

\begin{proof}[Proof of Theorem~\ref{Thm:ideal}]
Using Corollary~\ref{Coro:BF/}, the proof amounts to find the suitable generating set for the ideal
$
I
$.
Given $\alpha,\beta\in\mathbb{F}_n$,
we introduce the element
$$
\Theta({\alpha,\beta})=t_{\alpha\beta}-t_{\alpha},
$$
so that
\[
I=\left( \Theta(\alpha,\beta)\mid  \alpha,\beta\in\mathbb{F}_n,\ \beta\in\ker(\pi)\right).
\]
The proof consists in finding  a generating set for the ideal $I$ according to the statement of the theorem.  
Firstly, as $\ker(\pi)$ is normally
generated by the relations $r_l$, $l\in L$, by using the equalities
\begin{alignat*}{2}
 \Theta(\alpha,\beta_1\beta_2) & = \Theta(\alpha\beta_1,\beta_2)+ \Theta(\alpha,\beta_1), &&\qquad \forall \alpha,\beta_1,\beta_2\in\mathbb F_n, \\
 \Theta(\alpha,\gamma\beta\gamma^{-1} )& = \Theta(\gamma^{-1}\alpha\gamma,\beta), &&\qquad \forall \alpha,\beta,\gamma\in\mathbb F_n,
\end{alignat*}
we get:
\[
I=\left( \Theta(\alpha,r_l)\mid  \alpha\in\mathbb{F}_n,\ l\in L\right).
\]
For an element $\alpha\in \mathbb{F}_n$, let $|\alpha|$ denote 
its word length in the canonical generators. Define
\begin{equation}\label{eq:I_k}
I_k=\left( \Theta(\alpha,r_l)\mid  \alpha\in\mathbb{F}_n,\ \vert\alpha\vert \leq k, \ 
l\in L\right).
\end{equation}
so that $I=\bigcup_k I_k$.
We claim that
\begin{equation}
 \label{eqn:I2=I3}
I_{k+1}=I_k,\quad \textrm{for }k\geq 2 .
\end{equation}
 For that purpose we use an equality due to Vogt  \cite[Lemma 4.1.1]{FicoMontesinos}:
\begin{align}\label{eqn:Vogt}
2 t_{abcd} & =  t_{a} t_{b} t_{c} t_{d} - t_{c} t_{d} t_{ab} - t_{b} t_{c} t_{ad} - t_{a} t_{d} t_{bc} - t_{a} t_{b} t_{cd} \notag\\
&\quad+ t_{ad} t_{bc} - t_{ac} t_{bd} + t_{ab} t_{cd} + t_{d} t_{abc} + t_{c} t_{abd} + t_{b} t_{acd} + t_{a} t_{bcd}\,.
\end{align}
From this equation, we easily deduce 
\begin{multline}
\label{eqn:thetaabcd}
2\Theta(abc,d)=
(t_{abc} -t_a t_{bc}-t_ct_{ab} + t_at_bt_c)\Theta(1,d)\\+(t_{ab}-t_{a}t_{b})\Theta(c,d)
-t_{ac}\Theta(b,d)+(t_{bc}-t_bt_c)\Theta(a,d)\\
+t_a\Theta(bc,d)+t_b\Theta(ac,d)+t_c\Theta(ab,d),
\end{multline}
which implies \eqref{eqn:I2=I3}.

As $I=I_2$, $I$ is generated by $\Theta(1,r)$, $\Theta(\gamma_i^{\pm 1},r)$, and $\Theta(\gamma_i^{\pm 1}\gamma_j^{\pm 1},r)$,
so we just need to get rid of the powers $- 1$ and to reduce to $i<j$.
This is proved by using the equalities 
\begin{align}
 \label{eqn:thetaalphar} \Theta(\alpha, r)& =  t_\alpha\Theta(1,r)-\Theta(\alpha^{-1},r), \\
 \label{eqn:thetaalphabetar} \Theta(\alpha\beta,r)&=  t_\alpha\Theta(\beta,r)-\Theta(\alpha^{-1}\beta,r),
\end{align}
(the first one is obviously a particular case of the second one),
which in its turn follow from the trace identity $t_at_b=t_{ab}+t_{a^{-1}b}$ .
\end{proof}

\begin{proof}[Proof of the Addendum]
To simplify notation, assume that there is only one relation $r$ and decompose it as $r=u^{-1}v$. We prove the equalities
\begin{equation}
\label{eqn:addendum}
\left( t_{ar}-t_{a}\mid a\in \mathbb F_ n\right)   =  \left( t_{bv}-t_{bu}\mid b\in \mathbb F_ n\right) 
  = \left( t_{bv}-t_{bu}\mid b\in\{1,\gamma_i,\gamma_i\gamma_j\}\right).
\end{equation}
The first equality is elementary by writing $a=b u$. The second equality follows from 
\[
 t_{bv}-t_{bu}=\Theta(b,v)-\Theta(b,u)
\]
by applying the same arguments as in the proof of Theorem~\ref{Thm:ideal}, in particular equalities
\eqref{eqn:thetaabcd}, \eqref{eqn:thetaalphar}, and  \eqref{eqn:thetaalphabetar}, allow to reduce the word length of $b$.
Then the addendum follows from \eqref{eqn:addendum}. 
\end{proof}

\begin{Remark}\label{rem:ideal}
The ideal $I_1$ defined in \eqref{eq:I_k} is the ideal considered by Gonz\'{a}lez-Acu\~{n}a and Montesinos-Amilibia  \cite{FicoMontesinos}. 
They proved that the coordinate ring of the representation variety (not the scheme) is isomorphic to 
\[ \big(B(\mathbb{F}_n) / I_1\big)_{red}\cong B(\mathbb{F}_n)/ \operatorname{rad}(I_1)\,.
\]
In Section~\ref{sec:examples} we see that the ideal $I_1$ may 
be non-radical even if the coordinate ring $B(\Gamma)$ is reduced
(whether the ideal $I_1$ is a radical ideal or not may  
depend on the presentation of $\Gamma$).
Thus, when we compute  a scheme, $I_1$  is not sufficient and  we need to 
consider $I_2$.
The following example illustrates that the ideals $I_2$ 
and $I_1$ may be different, even when $B(\Gamma)$ is reduced.
\end{Remark}

 \begin{Example}\label{ex:trefoil}
 Let $\Gamma = \langle a,b\mid r \rangle$, 
 $r=baba^{-1} b^{-1} a^{-1}$,  be the trefoil group.
The algebra $B(\Gamma)$ is a quotient of 
 $B(\mathbb{F}_2) \cong \CC[t_{a},t_{b},t_{ab}]$.
As  $\Gamma$ is a one-relator group, we obtain:
\[
I_{1}=(  t_{r}-2, \ t_{a\, r }-t_{a},\ 
 t_{b\, r }-t_{b})
\  \text{ and }\ 
I_{2}=(  t_{r}-2, \ t_{a\, r }-t_{a},\ 
 t_{b\, r }-t_{b},\ t_{ab\, r }-t_{ab})\,.
\]

Let's compute $I_1$.
By using the trace relations we obtain $ t_{a \, r} =  t_{b}$ and hence 
$
t_{a\, r }-t_{a}=t_{b}-t_{a}
$.
Moreover, we have
\[
t_{b\, r} =  t_{b} t_{r}-  t_{r\, b^{-1}} \quad\text{ and }\quad
t_{r\, b^{-1}} = t_{a}\,.
\]
Therefore, 
\[
t_{b\, r} - t_{b} = t_{b} t_{r}-  t_{a} -t_{b} = t_{b}(t_r -2) + (t_{b} - t_{a}),
\]
and we obtain $I_{1} = \left( t_r-2, t_{a} - t_{b}\right)$.
The trace relations \eqref{eq:trace_rel} give:
\begin{align*}
 t_{baba^{-1} b^{-1} a^{-1}} & = 
 t_{aba}t_{b a b} - t_{(ba)^3}\\
 t_{aba} &= t_{ab} t_{a} - t_{b}\\
 t_{bab} &= t_{ab} t_{b} - t_{a}\\
 t_{(ba)^3} &= t_{ab}^3 -3\,t_{ab}\,,
 \end{align*}
 and therefore
 \[
 t_r-2 = (t_{ab} t_{a} - t_{b}) (t_{ab} t_{b} - t_{a}) - t_{ab}^3 +3\,t_{ab}\,.
 \]
We obtain the following primary decomposition of $I_{1}$:
 \[
 I_{1} = (t_{a} - t_{b}, t_{b}^2 - t_{ab} - 2) \cap 
 (t_{a} - t_{b}, (t_{ab}-1)^2)
 \]
 and $I_{1}$ is clearly not radical (see also \cite{notebook}). Notice that the first ideal corresponds to the characters of non-simple representations 
 $t_{ab^{-1}}-2 = t_{a}t_{b} - t_{ab} -2$.
%
 
 On the other hand, we have that $ab\, r = (ab)b a ba^{-1}(ab)^{-1}$, and  hence
$t_{ab} = t_{b a ba^{-1}}$ 
modulo $I_2$. Now, 
\begin{align*}
t_{ab\, r}-t_{ab} &= t_{b a ba^{-1}} -t_{ab} \\
 &= t_{b a b}t_{a} - t_{(ba)^2}-t_{ab} \\
&= (t_{ab} t_{b} - t_{a}) t_{a} - 
(t_{ab}^2-2)-t_{ab} 
\end{align*}
We obtain the following primary decomposition of $I_{2}$:
 \[
 I_{2} = (t_{a} - t_{b}, t_{b}^2 - t_{ab} - 2) \cap 
 (t_{a} - t_{b}, t_{ab}-1)
 \]
 and $I_{2}$ is radical (see also the notebook \cite{notebook}).
%
\end{Example}

\begin{Remark}
We observe also that $\CC[t_{\gamma_1},t_{\gamma_2},t_{\gamma_1\gamma_2}]/I_{1}$ depends on the presentation and not only on the group.
An explicit example is given by an other presentation of the trefoil group (the group of Example~\ref{ex:trefoil}): the presentation $\langle x, y\mid x^2y^{-3}\rangle$
produces a radical ideal $I_{1}$ (see \cite{notebook} for the computations).
\end{Remark}
  
\begin{Remark}
The ideal $I_3$ is the ideal considered by Miura and Suzuki in \cite{MiuraSuzuki2021}.  They proved that the skein module of a group with three generators and 
two relations (and hence the coordinate ring of the representation scheme) is isomorphic to $B(\mathbb{F}_3) / I_3$. As pointed out in the proof of Theorem~\ref{Thm:ideal}, we have in general that $I_2=I_3$, see  equation~\eqref{eqn:I2=I3}, and therefore it is sufficient to consider less generators of the ideal.
\end{Remark}

\begin{Remark}
Our result was mainly motivated by Corollary~10.1.7 in \cite{Fogg}. This corollary states that any product of matrices (and their inverses) constructed from a given set 
$\{A_1,\ldots,A_n\}\subset\SL2$ can be written as a linear combination of $\id$, $A_i$, $1\leq i\leq n$, and $A_iA_j$, $1\leq i < j\leq n$.
\end{Remark}
 
\section{The scheme of non-simple representations}
\label{sec:nonsimple}

 In concordance with \cite{LubotzkyMagid} we use the term \emph{simple} 
 representation instead of the more common term \emph{irreducible} 
 representation in order to avoid confusion with irreducible algebraic 
 varieties or schemes, as well as reduced or non-reduced schemes.  
 Thus the term \emph{non-simple} in the title of this section would be \emph{reducible} in other papers.
 
\begin{Definition}[\cite{LubotzkyMagid}]
 Let $\rho\co\Gamma\to\SLL_2(\CC)$ be a representation. We say that $\rho$ is 
 \emph{simple}
if the image $\rho(\Gamma)$ spans $M_2(\CC)$ as a $\CC$-vector space.
 \end{Definition}

\begin{Lemma}\label{lemma:123-132}
 Let $\rho\co\Gamma\to\SLL_2(\CC)$ be a representation. Then $\rho$ is simple if and only if there exist
 $\gamma_1$, $\gamma_2$, $\gamma_3$ in $\Gamma$ such that
 \[
  \tr\big(\rho(\gamma_1)\rho(\gamma_2)\rho(\gamma_3)\big)-
  \tr\big(\rho(\gamma_1)\rho(\gamma_3)\rho(\gamma_2)\big) \neq 0\,.
 \] 
\end{Lemma}

\begin{proof}
Let $\id\in\SL2$ denote the identity matrix.
 For a matrix $M\in M_2(\mathbb{C})$ we put $M_0 = M-\frac{\tr(M)}2 \id$. Then $\id$, $A$, $B$ and $C$ generate $M_2(\mathbb{C})$ as a $\mathbb{C}$-vector 
 space if and only if $(A_0,B_0,C_0)$ is a $\mathbb{C}$-basis for the subspace $\sll_2(\mathbb{C})$ of trace free matrices of $M_2(\mathbb{C})$.
 Now observe that the map $\sll_2(\mathbb{C})^3\to \mathbb{C}$ given by
 \begin{equation}
\label{eqn:det}
  (X_0,Y_0,Z_0)\mapsto \tr(X_0 Y_0 Z_0)
 \end{equation}
 is a determinant map, 
 e.g.~multilinear and alternating 
 (see Section~V in \cite{Fenchel}).
 
 A direct calculation gives
 \begin{multline}\label{eq:ABC-ACB}
  \tr(ABC) = \tr(A_0B_0C_0) -\frac12
  \big( \tr(A)\tr(BC) \\ + \tr(B)\tr(AC) + \tr(C)\tr(AB) -
  \tr(A)\tr(B)\tr(C)
  \big)\,,
 \end{multline}
 and hence
 \[
  \tr(ABC)-\tr(ACB) = 2\,\tr(A_0B_0C_0)\,.
 \]
 Thus $\tr(ABC)\neq \tr(ACB) $ if and only if 
 $A_0, B_0, C_0$ are linearly independent (equivalently, they span $\sl2$). As $\id\in\rho(\Gamma)$, the lemma follows.
\end{proof}
\begin{Remark}
 Notice that equation~\eqref{eq:ABC-ACB} and the fact that \eqref{eqn:det}
 is a determinant map imply the equality $t_{abc}+t_{acb}= p$ 
 in \eqref{eq:p_and_q}.
\end{Remark}

Motivated by  Lemma~\ref{lemma:123-132} we define:
\begin{Definition}
Let $\Gamma$ be a group, and $B(\Gamma)$ the universal algebra of $\SL2$-characters.
The ideal $\ns{J}\subset B(\Gamma)$ is defined by 
\[
 \ns J= \ns J (\Gamma) := \big( t_{abc} - t_{acb} \mid a,b,c \in \Gamma \big)\,, 
\]
and the quotient
\[
 \ns B(\Gamma) = B(\Gamma)/\ns J
\]
is called the universal algebra of non-simple $\SL2$-characters.
\end{Definition}

\begin{Remark}
 We have 
 \[
  \ns J\subset B(\Gamma) = A(\Gamma)^{\SL2}\subset A(\Gamma)\,,
 \]
and Lemmata 3.4.1, 3.4.2 and Remark~3.4.3 in \cite{Newstead} give that
\begin{align*}
 \ns J A(\Gamma) \bigcap B(\Gamma) &= \ns J\\
 \intertext{ and }
 \big( A(\Gamma)/ \ns J A(\Gamma)\big)^{\SL2} & \cong
 A(\Gamma)^{\SL2} / \ns J = \ns B (\Gamma)\,.
\end{align*}
\end{Remark}

For an abelian group $\Gamma_{0}$ we have $\ns J(\Gamma_{0}) = (0)$  since $abc=acb$ in $\Gamma_0$, and hence $\ns B(\Gamma_{0}) = B(\Gamma_{0})$.
 
\begin{Proposition}
\label{prop:abelianize}
Let $\Gamma$ be a finitely generated group. 
The abelianization morphism $\Gamma\to\Gamma_{ab}$ 
induces an isomorphism of algebras $ B(\Gamma_{ab}) \xrightarrow{\cong} \ns B (\Gamma)$.
\end{Proposition}

\begin{proof}
By writing $c=b^{-1}a^{-1}d$, we have an equality of sets
 $$
  \{t_{abc} - t_{bac} \mid a,b,c \in \Gamma\}= \{t_{d} - t_{[b,a]d} \mid a,b,d \in \Gamma\}
 $$
So both sets generate the same ideal in $B(\Gamma)$. The left hand side set spans  $\ns J$, the right hand side set
spans the ideal of the abelianization, by Theorem~\ref{Thm:ideal}.
\end{proof}

\section{Abelian groups}
\label{sec:abelian}

Throughout this  section   $\Gamma$ denotes a finitely generated abelian group.
Przytycki and Sikora proved that the skein algebra of an abelian group is 
reduced \cite{PrzytyckiSikoraBanach,PrzytyckiSikora},  hence its character scheme  is also reduced.
In \cite{Sikora_abelian} Sikora described the variety of characters when $\Gamma$ is torsion free.
Here we discuss the case with torsion, and we show in particular that $B(\Gamma)$ is reduced with completely different methods.

Let $\beta$ denote the first Betti number of $\Gamma$ and $T$ its torsion subgroup, so that 
we have a short exact sequence 
$$
1\to T\to \Gamma\to  \mathbb{Z}^\beta\to 1
$$
with $T$ finite. The
scheme of characters of $T$ is a variety with finite cardinality. 
For each $\chi\in X(T,\SL2)$,
set 
$$
X(\Gamma,\SL2)_\chi=\mathrm{res}^{-1} (\chi)
$$
where $\mathrm{res}\colon X(\Gamma,\SL2)\to X(T,\SL2)$ is the map induced by restriction.

\begin{Theorem}
\label{thm:abelian} For $\Gamma$ an abelian group as above, 
$X(\Gamma,\SL2)$ is reduced and 
 $$
 X(\Gamma,\SL2)=\bigcup\limits_{ \chi\in X(T,\SL2) }X(\Gamma,\SL2)_\chi
 $$
 is its decomposition into irreducible components. Furthermore
 $$
 X(\Gamma,\SL2)_\chi\cong\begin{cases}
                     (\CC^*)^\beta & \textrm{if }\chi \textrm{ is not central}\\
                     (\CC^*)^\beta/\sim & \textrm{if }\chi \textrm{ is  central}
                    \end{cases}
 $$
 where 
$(\lambda_1,\ldots,\lambda_\beta)\sim (\frac{1}{\lambda_1}, \ldots,\frac{1}{\lambda_\beta})$,
for $\lambda_1,\ldots,\lambda_\beta\in   \CC^* $.
\end{Theorem}

When $\chi\in X(T,\SL2)$ is central, a point in $X(\Gamma,\SL2)_\chi$ is the character of a representation
that maps the generators of  $\mathbb Z^\beta$ to diagonal matrices with eigenvalues 
$\lambda_i^{\pm 1}$. Thus the action of the involution is the action of the Weyl group.
In particular:

\begin{Corollary}
The singular locus of $X(\Gamma,\SL2)$ is the set of  central characters of $\Gamma$. 
\end{Corollary}

To prepare the proof of Theorem~\ref{thm:abelian}, we need a lemma on the Zariski tangent space.

\begin{Lemma}\label{lemma:dimz1}
 If $\rho\in R(\Gamma,\SL2)$ is non-central, then $\dim Z^1(\Gamma,\mathrm{Ad}\rho)\leq \beta+2$.
\end{Lemma}

\begin{proof}[Proof of Lemma~\ref{lemma:dimz1}]
We may assume $\beta\geq 1$, as the finite case has been considered in Example~\ref{ex:finite}.
Chose elements  $ \gamma_1,\ldots, \gamma_\beta\in \Gamma$ so that they project to a generating set of 
$\Gamma/T\cong  \mathbb{Z}^\beta$. 
This choice  yields a splitting  $\Gamma\cong \mathbb{Z}^\beta\times T$. 
To bound the  dimension of the space of 1-cocycles, consider the monomorphism of vector spaces:
\begin{equation}
 \label{eqn:embedd}
\begin{array}{rcl}
  Z^1(\Gamma,\mathrm{Ad}\rho) & \to & \sl2\times\overset{(\beta)}\cdots \times \sl2\times Z^1(T,\mathrm{Ad}\rho\vert_T)\\
  d &\mapsto & (d ( \gamma_1),\ldots, d ( \gamma_\beta), d\vert_T)
\end{array} 
\end{equation}
where $\rho\vert_T$ and $d\vert_T\in Z^1(T,\mathrm{Ad}\rho\vert_T)$ denote the respective restrictions to $T$ 
of $\rho$ and  $d$. 
Since $T$ is finite, $H^1(T,\mathrm{Ad}\rho\vert_T)=0$ and therefore
\begin{equation}
 \label{eqn:dimZ1T}
\dim Z^1(T,\mathrm{Ad}\rho\vert_T)=B^1(T,\mathrm{Ad}\rho\vert_T)=\begin{cases}
                                         0 & \textrm{if }\rho\vert_T \textrm{ is central}\\
                                         2 & \textrm{otherwise}
                                        \end{cases}
\end{equation}
because $\dim B^1(T,\mathrm{Ad}\rho\vert_T)=\dim\sl2-\dim \sl2^{\mathrm{Ad}\rho(T)}$.

Next we distinguish cases, according to whether the restriction $\rho\vert_T$ is central or not.

When the restriction
$\rho\vert_T$ is central, we may assume that $\rho( \gamma_1)$ is non-central. Then for every
$d\in Z^1(\Gamma,\mathrm{Ad}\rho)$, from  $ \gamma_1 \gamma_i= \gamma_i \gamma_1$ we get (following the rules of crossed homomorphism or Fox calculus):
\begin{equation}
\label{eqn:comm}
 (\mathrm{Ad}_{\rho(\gamma_1)}-{\id})d(\gamma_i)=
 (\mathrm{Ad}_{\rho(\gamma_i)}-{\id})d(\gamma_1)
\end{equation}
As $\rho(\gamma_1)$ is non-central, 
$
 \operatorname{rank} 
 (\mathrm{Ad}_{\rho(\gamma_1)}-{\id}) =2
$ (because the kernel of $(\mathrm{Ad}_{\rho(\gamma_1)}-{\id}) $
is the tangent line to the 1-parameter group containing  $\rho(\gamma_1)$). 
Hence, 
from \eqref{eqn:comm} applied to $i=2,\ldots,\beta$ and using \eqref{eqn:dimZ1T},
the image of \eqref{eqn:embedd} is contained in the kernel of a linear map
$\sl2^{\beta}\to  \sl2^{\beta-1}$ of rank $ 2(\beta-1)$, and we get the bound
$\dim Z^1(\Gamma,\mathrm{Ad}\rho)\leq 3\beta-2(\beta-1)$.

When the restriction $\rho\vert_T$ is non-central chose $\gamma_0\in T$ such that $\rho(\gamma_0)$ 
is non-central. 
Then apply the same argument as above to the equalities $\gamma_0\gamma_i=\gamma_i\gamma_0$,
for $i=1,\ldots,\beta$. Now the image of 
\eqref{eqn:embedd} is contained in the kernel of a linear map
$\sl2^{\beta}\oplus \CC^2\to  \sl2^{\beta}$ of rank $2 \beta $, hence  
$\dim Z^1(\Gamma,\mathrm{Ad}\rho)\leq 3\beta+2-2\beta$.
\end{proof}

 For a morphism $\tau\colon T\to \mathbb Q/\mathbb Z$ and a splitting $\Gamma\cong \mathbb Z^\beta\times T$, define 
 the subset of representations 
 $V_\tau \subset (R(\Gamma,\SL2)_{\textrm{red}}$ as:
 \begin{equation}
  \label{eqn:vtau}
   V_\tau=\{ \exp( \theta a)\exp(2\pi \tau a)\mid \theta\in\hom(\Gamma,\CC)  , \   a\in \mathcal Q  \},
 \end{equation} 
 where $\mathcal Q$ denotes the quadric
 $$\mathcal Q=\{  a\in\sl2\mid \det(a)=1\}=\sl2\cap\SL2.$$
 To see that the set $ V_\tau$ consists of representations, notice that $ \exp(2\pi a)=
{\id}$ for every $a\in \mathcal Q$.
The splitting $\Gamma\cong \mathbb Z^\beta\times T$ is used to extend 
$\tau\colon T\to \mathbb Q/\mathbb Z$ to the whole $\Gamma$, but the set $ V_\tau$
does not depend on this extension or splitting.

We shall use $V_\tau$ to describe the preimage of  $X(\Gamma,\SL2)_\chi$
by the natural projection
$R(\Gamma,\SL2)\to X(\Gamma,\SL2)$, that we denote by 
$$R(\Gamma,\SL2)_\chi =
\{\rho\in R(\Gamma,\SL2) \mid \chi_{\rho\vert_T}=\chi\} .
$$


Consider $\chi\in X(T,\SL2)$ non-central. 
Every representation of $T$
with character $\chi$ is of the form $\exp(2\pi\tau a)$, for some $a \in \mathcal Q$ and some  
$\tau\colon T\to \mathbb Q/\mathbb Z$. 
Notice that for $\chi\in X(T,\SL2)$ non-central, we can chose either $\tau$ or $-\tau$, 
but this is the only ambiguity in the choice of $\tau$. 
In fact  $V_\tau$= $V_{\tau'}$ if and only if $\tau=\pm \tau'$.

\begin{Lemma}
\label{lemma:non-central}
For $\chi\in X(T,\SL2)$ non-central, and $\tau$ as above, 
$R(\Gamma,\SL2)_\chi$ is reduced and
 \[
 R(\Gamma,\SL2)_\chi= V_\tau\cong (\CC^*)^\beta\times \mathcal Q.
 \]
In particular 
$R(\Gamma,\SL2)_\chi$ is irreducible and smooth. 
\end{Lemma}

\begin{proof}
As
all abelian subgroups of $\SL2$ are contained in a one-parameter group,
we have equality of sets
$
(R(\Gamma,\SL2)_\chi)_{\textrm{red}}= V_\tau
$. 
Using  Lemmas~\ref{lemma:schemesmooth} and~\ref{lemma:dimz1}, since 
$\dim V_\tau=\beta+2$ every point in $R(\Gamma,\SL2)$ is scheme smooth (hence reduced), and 
$(R(\Gamma,\SL2)_\chi)_{\textrm{red}}=R(\Gamma,\SL2)_\chi$. 

Finally, notice that the natural map 
$ (\CC^*)^\beta\times \mathcal Q\to V_\tau$ 
is a bijection because $\exp(2\pi\tau a)$ is non-central.
By construction this map  is regular, and by smoothness and Zariski's main theorem 
\cite[3.20]{Mumford} it is biregular.
\end{proof}

If $\chi\in X(T,\SL2)$ is central then $R(\Gamma,\SL2)_\chi$ contains parabolic representations, and we 
need other tools to study this case.

\begin{Definition}
 The \emph{cone of group homomorphisms of rank at most one} 
from $\Gamma$ 
to $\sl2$ is the determinantal variety:
$$
\hom_1(\Gamma, \sl2)\cong \hom_1(\Gamma/T, \sl2)\cong \hom_1(\CC^\beta,\sl2).
$$
\end{Definition}
Namely,  $\hom_1(\Gamma, \sl2)$ consists of all group homomorphisms
in $\hom(\Gamma, \sl2)$ whose image is contained in a linear space of dimension at most 1.
This is the $\CC$-cone on the Segre variety 
$\mathbb{P}^{\beta-1}\times \mathbb{P}^2$.

Now consider a central character $\chi\in X(T,\SL2)$.
We chose  $\rho_0\in R(\Gamma,\SL2)_\chi$ a central representation of $\Gamma$ 
whose restriction to $T$ has character $\chi$.

\begin{Lemma}
\label{lemma:central}
For $\chi\in X(T,\SL2)$ central, 
and given $\rho_0\in R(\Gamma,\SL2)_\chi$ also central, 
 \[
 (R(\Gamma,\SL2)_\chi)_{\textrm{red}}= \{ \exp(\theta)\rho_0\mid \theta\in \hom_1(\Gamma, \sl2)  \}.
 \]
Furthermore
$R(\Gamma,\SL2)_\chi$ is scheme-smooth at non-central representations. 
\end{Lemma}

\begin{proof}
As in Lemma~\ref{lemma:non-central},
equality $(R(\Gamma,\SL2)_\chi)_{\textrm{red}}= \{ \exp(\theta)\rho\mid \theta\in \hom_1(\Gamma, \sl2)$ as sets
 is a consequence of the fact that every abelian subgroup of $\SL2$ is contained in a one-parameter group,
 which in its turn is the image by the exponential of a line in $\sl2$. 
 Scheme smoothness at non-central representations follows again from Lemmas~\ref{lemma:schemesmooth} and~\ref{lemma:dimz1}
 and the dimension count.
\end{proof}

\begin{Lemma}
\label{Lemma:QuadraticCone}
The cone $\hom_1(\Gamma, \sl2)$ is isomorphic to the quadratic cone 
to $R(\Gamma,\SL2)$ at 
any central
representation and it is the whole tangent cone.
\end{Lemma}

\begin{proof}
 Let  $\rho$ be a central representation. In particular $ \mathrm{Ad}\,\rho $ is trivial and
 \[
 Z^1(\Gamma,\mathrm{Ad}\rho)=\hom(\Gamma, \sl2)
 \]
 (namely, group homomorphisms with no restriction on the dimension of the image).
 To compute the quadratic cone, notice that, by
 Baker-Campbell-Hausdorff formula,
 for any $\theta \in \hom(\Gamma, \sl2)$:
 \[
  e^{t \theta(\gamma_1)} e^{t \theta(\gamma_2)} e^{-t \theta(\gamma_1)} e^{-t \theta(\gamma_2)}
  = e^{t^2 [\theta(\gamma_1), \theta(\gamma_2)]+ O(t^3)}, \qquad \forall \gamma_1,\gamma_2\in\Gamma.
 \]
This yields that $\hom_1(\Gamma, \sl2)$ is the quadratic cone at $\rho$. (This can be viewed equivalently with the 
obstruction theory of Goldman.)
Not only $\hom_1(\Gamma, \sl2)$ is the quadratic cone, by construction
it is also the whole tangent cone, because all higher order terms in Baker-Campbell-Hausdorff formula vanish when 
$ [\theta(\gamma_1), \theta(\gamma_2)]=0$. 
\end{proof}

Related to Lemma~\ref{Lemma:QuadraticCone}, notice that a theorem of Goldman and Millson \cite[Thm.~9.3]{GoldmanMillson} guarantees that the singularities
of $R(\Gamma,\SL2)$ are at most quadratic,  as $\Gamma$ is virtually a Bieberbach group.

Since $\hom_1(\Gamma, \sl2)$ is a determinantal scheme 
 it is reduced and normal \cite{BrunsVetter},
and by Lemma~\ref{lem:reduced}, we get:

\begin{Corollary}
\label{cor:centralreduced}
 A central representation is a reduced and normal point   of the scheme
 $R(\Gamma,\SL2)$.
\end{Corollary}

From the previous results in this section we deduce:

\begin{Corollary} The scheme $R(\Gamma,\SL2)$ is reduced and normal. Its singular
locus is precisely the set of central representations, that have quadratic singularities modeled in the $\CC$-cone on 
$\mathbb{P}^{\beta-1}\times \mathbb{P}^2$.
\end{Corollary}

\begin{proof}[Proof of Theorem~\ref{thm:abelian}] As
$R(\Gamma,\SL2)$ is reduced and normal, and its irreducible components
 are $R(\Gamma,\SL2)_\chi$,  $X(\Gamma,\SL2)$ is also reduced and normal, and is components are the
$X(\Gamma,\SL2)_\chi$.

Fix $a\in \mathcal{Q}$ and consider 
\begin{equation*}
\begin{array}{rcl}
 \tilde \varphi_a\colon\hom(\Gamma,\CC)\cong\CC^\beta&\to& R(\Gamma,\SL2)_\chi \\
 \theta&\mapsto  &  \exp(\theta a)\rho_0
\end{array}
\end{equation*}
where $\rho_0\in R(\Gamma,\SL2)_\chi$ is a central representation when $\chi$
is central as in Lemma~\ref{lemma:central},
or $\rho_0=\exp(2\pi \tau a)$ for some nontrivial $\tau\colon T\to \ZZ/\mathbb{Q}$ as in~\eqref{eqn:vtau}
when $\chi$ is not central.
The map
 $\tilde \varphi_a$ factors through the exponential map on each factor $\CC$ to 
\begin{equation*}
\varphi_a\colon (\CC^*)^\beta\to R(\Gamma,\SL2)_\chi.
\end{equation*}
We compose it with the projection $\pi\colon R(\Gamma,\SL2)_\chi\to X(\Gamma,\SL2)_\chi$ and 
we
claim that the composition 
\begin{equation}
\label{eqn:bijnc}
\pi\circ\varphi_a \colon (\CC^*)^\beta\to X(\Gamma,\SL2)_\chi 
\end{equation}
is a surjection. When $\chi$ is not  central, the claim follows
from  Lemma~\ref{lemma:non-central} and the description
of $V_\tau$ in equation~\eqref{eqn:vtau}, as $\SL2$ acts transitively on
$\mathcal{Q}$ (hence every orbit by conjugation in $R(\Gamma,\SL2)_\chi$
meets the image of $\varphi_a$). When $\chi$ is central, we use the description of
Lemma~\ref{lemma:central} and in this case we similarly prove that 
 every orbit by conjugation in $R(\Gamma,\SL2)_\chi$ of a  semi-simple representation
meets the image of $\varphi_a$. We conclude the surjectivity by recalling that 
every character is the character of a semi-simple representation.

Next we discuss the inverse images of the surjection $ \pi\circ\varphi_a $ in \eqref{eqn:bijnc}. We use
two elementary properties:
\begin{itemize}
 \item If two semisimple representations have the same character, then they are conjugate.
 Notice that the image of $\varphi_a$ consists of only semisimple representations.
 \item We write a diagonal representation of $\Gamma$ as $\mathrm{diag}(\theta,\theta^{-1})$ for
 $\theta\colon\Gamma\to\mathbb{C}^*$ a homomorphism. If  $\mathrm{diag}(\theta_1,\theta_2^{-1})$ is conjugate to
 $\mathrm{diag}(\theta_2,\theta_2^{-1})$, then $\theta_1 = \theta_2^{\pm 1}$.
\end{itemize}
To apply these remarks to our situation,  we may assume that $a\in\mathcal{Q}$ is diagonal. 
Now, if $\tau$ is central, then it follows  that   
$ \pi\circ\varphi_a $ in \eqref{eqn:bijnc}
factors to the quotient 
as in the statement of the theorem
\begin{equation}
\label{eqn:bijc}
 (\CC^*)^\beta/\sim\to X(\Gamma,\SL2)_\chi
\end{equation}
which is a bijection.
This does not hold anymore when 
$\tau$ is non central, as the restriction to $T$
of a representation in the image of $\varphi_a$
is fixed, it is a representation 
$\rho_0=\exp(2\pi \tau a)$ with $\tau\colon T\to \ZZ/\mathbb{Q}$ nontrivial. We conclude that  when 
$\tau$ is non central 
\eqref{eqn:bijnc} is already a bijection.

The maps \eqref{eqn:bijnc} and \eqref{eqn:bijc} are regular  bijections, and by Zariski's main theorem \cite[3.20]{Mumford}
they are both biregular,
since $X(\Gamma,\SL2)_\chi$
is normal.
\end{proof}

Next we describe the algebra $B(\ZZ^\beta)=\CC[ X(\ZZ^\beta,\SL2)]$ in terms of traces. Start with the presentation
$$
\ZZ^\beta=\langle a_1,\ldots ,a_\beta\mid [a_i, a_j]=1\rangle
$$

\begin{Proposition}
\label{Prop:zbeta}
With the previous presentation, 
$$
B(\ZZ^\beta)= \CC[t_{a_i}, t_{a_ia_j}]/(t_{a_ia_j}-t_{a_ja_i},f_{a_i, a_j},g_{a_i, a_j, a_k} , h_{a_i,a_j,a_k,a_l})_{i,j,k,l\in \{1,\ldots,\beta\}}
$$
where  the subindex are  different on each $  f_{a_i, a_j},g_{a_i, a_j, a_k} , h_{a_i,a_j,a_k,a_l}  $,  and 
\begin{align}
\label{eqn:relf} f_{a,b}= &  t_{a}^2 +t_{b}^2+t_{ab}^2-t_{a}t_{b}t_{ab}-4 ,\\
\label{eqn:relg} g_{a,b,c}= & t_{a}( t_{a}t_{bc}+t_{b}t_{ac}
+t_{c}t_{ab}-t_{a}t_{b}t_{c}   )  - 2 t_{ab}t_{ac}-4t_{bc} +2 t_bt_c,\\
\label{eqn:relh} h_{a,b,c,d}= & (2t_{ab}-t_at_b)(2t_{cd}-t_ct_d)- (2t_{ac}-t_at_c)(2t_{bd}-t_bt_d).
\end{align}
Furthermore, for computing trace functions of elements of $\ZZ^\beta$, we use the standard trace identities \eqref{eq:trace_rel} and 
\eqref{eqn:Vogt} and add
\begin{equation}
\label{eqn:t{abc}}
 2 t_{abc}=t_{a}t_{bc}+t_{b}t_{ac}
+t_{c}t_{ab}-t_{a}t_{b}t_{c} .
\end{equation}
\end{Proposition}

\begin{Remark}
\label{rem:Zbeta}
 \begin{enumerate}[(a)]
  \item  Equation~\eqref{eqn:t{abc}} does not need to be included in the presentation of the algebra  $B(\ZZ^\beta)  $, 
  but it is needed to describe the trace functions of group elements. 
  It can be read as $t_{abc}=\frac{1}{2}p$, where $p$ is as in \eqref{eq:p}.
  Notice that 
  from the relations \eqref{eqn:relf}, \eqref{eqn:relg} and \eqref{eqn:relh}
  it follows that $p^2- 4 q=0$, so the polynomial $F$ in  \eqref{eq:F} reads as  $F=(t_{abc}-\frac{1}{2}p)^2$.  
  Vanishing of $F$ is a standard trace identity,  but  we include \eqref{eqn:t{abc}} to get rid of the square. 
  \item Using \eqref{eqn:t{abc}}, \eqref{eqn:relg} reads as 
  $$
    g_{a,b,c}= 2( t_{a} t_{abc}  -  t_{ab}t_{ac}-2t_{bc} +t_bt_c  ).
  $$
  Notice also that, for $F$ 
  as in \eqref{eq:F}:  
  $$
  \tfrac{\partial F \phantom{t\,}}{\partial t_{bc}} = -t_{a} t_{abc}  +  t_{ab}t_{ac}+2t_{bc} -t_bt_c
  =-\tfrac{1}{2} g_{a,b,c}.$$
  \item  If we allow that some of the subindex are equal, then 
  we can reduce to a single family of equations, because:
  $$2 f_{a,b}=g_{a,b,b}\quad\textrm{ and }\quad 2g_{a,b,c}=h_{a,a,b,c}.
  $$
 \end{enumerate}
\end{Remark}

\begin{proof}[Proof of Proposition~\ref{Prop:zbeta}]
By Theorem~\ref{thm:abelian},
\[
B(\ZZ^\beta)= \left(\CC[\lambda_1,\ldots,\lambda_\beta, {\mu_1},\ldots,{\mu_\beta}] 
/ (\lambda_i\mu_i-1 )_{i=1,\ldots,\beta} \right )^\sigma
\]
where 
$\sigma(\lambda_1,\ldots,\lambda_\beta, {\mu_1},\ldots,{\mu_\beta})= 
({\mu_1},\ldots,{\mu_\beta},  \lambda_1,\ldots,\lambda_\beta)$.
A point in $X(\ZZ^\beta,\SL2)$ with coordinates   $(\lambda_1,\ldots,\lambda_\beta, {\mu_1},\ldots,{\mu_\beta})$ is
the character of the representation that maps each generator $a_i\in\mathbb Z^\beta$
to the diagonal matrix $\operatorname{diag}(\lambda_i, \frac{1}{\lambda_i})$.
%
%
To compute the invariant subalgebra, 
we change coordinates
\[
t_{a_i}=\lambda_i+\mu_i=\lambda_i+\tfrac1{\lambda_i}, \qquad 
x_i=\lambda_i-\mu_i=\lambda_i-\tfrac1{\lambda_i},
\]
so
\[
B(\ZZ^\beta)=\left( \CC[  t_{a_1},\ldots, t_{a_\beta},x_1,\dots, x_\beta ]/ (t_{a_i}^2-x_i^2-4 )_{i=1,\ldots,\beta} \right)^\sigma
\]
%
and $\sigma(t_{a_1},\ldots, t_{a_\beta},x_1,\dots, x_\beta)= 
(t_{a_1},\ldots, t_{a_\beta},-x_1,\dots, -x_\beta)$.
The algebra of invariants by $\sigma$  is generated by the $t_{a_i}$ and the quadratic monomials 
\[
z_{ij}=x_ix_j, \qquad i, j=1,\ldots,\beta.
\]
Thus $B(\ZZ^\beta)=\CC[t_{a_i},z_{ij}]/I$ where $I$ is the ideal generated by: 
\[
t^2_{a_i}-z_{ii}-4,\qquad z_{ij}-z_{ji}, \qquad z_{ij}z_{kl}-z_{il}z_{kj},
\]
for $i,j,k,l\in\{1,\ldots,\beta\}$, possibly equal.
The $z_{ij}$ can be written in terms of traces as follows:
\[
z_{ij}=x_ix_j= \lambda_i\lambda_j+\frac{1}{\lambda_i\lambda_j}-\frac{\lambda_i}{\lambda_j}-\frac{\lambda_j}{\lambda_i}  
=t_{a_ia_j}-t_{a_ia_j^{-1}}= 2t_{a_ia_j}-t_{a_i}t_{a_j}
\]
and the presentation of $B(\ZZ^\beta)$ as a quotient of $\CC[t_{a_i}, t_{a_ia_j} ]$ 
follows by writing the $z_{ij}$ in terms of traces.

Finally, notice that the unique relation we need to add to compute the trace of any element 
from these variables is \eqref{eqn:t{abc}}, see the discussion in Remark~\ref{rem:Zbeta}, because the trace of elements of 
length larger that three is deduced from Vogt relation \eqref{eqn:Vogt}.
\end{proof}

We show a couple of examples that illustrate 
Theorem~\ref{thm:abelian} when $\Gamma$ has torsion.

\begin{Example}\label{ex:Z+z4}
 We compute the scheme of characters of 
 \[
 \mathbb Z\oplus \mathbb Z/4\mathbb Z=\langle a,b\mid [a,b]=b^4=1\rangle .
 \]
 As $\mathbb Z\oplus \mathbb Z/4\mathbb Z$ is a quotient of $\ZZ^2$, by 
 Proposition~\ref{Prop:zbeta} the coordinates are  $(t_a,t_b,t_ {ab})$ and they satisfy
 \begin{equation}
  \label{eqn:tcomm}
   t_a^2+t_b^2+t_{ab}^2-t_at_bt_{ab}-4=0
 \end{equation}
Furthermore, as $b^4=1$, we get 
 $$
 t_b(t_b+2)(t_b-2)=0.
 $$ 
 Thus there are three components:
 \begin{itemize}
  \item  When $t_b=0$, the character restricted to $\mathbb Z/4\mathbb Z$ is non-central. By replacing $t_b=0$ in 
  \eqref{eqn:tcomm} we get
 $t_a^2+t_{ab}^2-4=0$, which can be rewritten as 
 \[
(t_a+ i t_{ab})(t_a- i t_{ab})=4,
 \]
 which is isomorphic to $\CC^*$.
 \item When $t_b=\pm 2$,  \eqref{eqn:tcomm}  becomes $(t_a\mp t_{ab})^2=0$, and since we know it is reduced, we can get rid of the square and so $t_a=\pm t_{ab} $. Namely, two lines, because $(\CC^*/\sim) \cong \mathbb C$.
 \end{itemize}
\end{Example}

\begin{Example}\label{ex:Z+ Z+z4}
 Next consider
  $$
 \mathbb Z\oplus  \mathbb Z\oplus \mathbb Z/4\mathbb Z=\langle a,b, c\mid [a,b]=[a,c]=[b,c]=c^4=1\rangle.
 $$
 The same considerations as in Example~\ref{ex:Z+z4}  yield that the coordinates for the scheme of characters 
 are 
 $ (t_a,t_b,t_c, t_{ab},t_{ac},t_{bc})$
 and it
 has three components:
 \begin{itemize}
 \item One component has equations:
$$\left. \begin{aligned}
  t_c=& 0\\
  t_a^2+t_{ac}^2-4=& 0 \\
  t_b^2+t_{bc}^2-4=& 0 \\
  -2 t_{ab} +t_a t_b-t_{ac} t_{bc}=&0  
 \end{aligned}
 \right\}
 $$
So it is isomorphic to $\CC^*\times \CC^*$.
 \item The remaining  two components have equations
  $$\left. \begin{aligned}
 t_c\pm 2=&0 \\ 
 t_{ac}\pm t_a=&0\\  
 t_{bc}\pm t_b=&0\\  
 t_a^2+t_b^2+t_{ab}^2-t_at_bt_{ab}-4=& 0
 \end{aligned}
 \right\}
 $$
 and are isomorphic to $X(\mathbb Z^2,\SL2)\cong (\CC^*\times \CC^*)/\sim$.
 \end{itemize}
\end{Example}

\section{Examples of 
non-reduced schemes of characters}\label{sec:examples}

Kapovich and Millson have proved in \cite{Kapovich-Millson}
that there are no restrictions on the local geometry of 
$\SL2$-character schemes of 3--manifold groups, in particular
there are non-reduced character schemes of those groups. 
In this section we give explicit examples of 
non-reduced $\SL2$-character schemes of 3--manifold and orbifold groups.

\begin{Example}
 \label{ex:fig8}
Let $S^3(K,3)$
denote the three-dimensional orbifold
with underlying space $S^3$, branching locus the 
figure-eight knot $K$ and branching index 3. It is a Euclidean
orbifold \cite[Ch.~13]{ThurstonNotes}. The Euclidean structure is relevant, because the translational part
of the Euclidean holonomy is a non-integrable infinitesimal deformation
of its rotational part. This is similar to the examples of Lubotzky and Magid in 
 \cite[pp.~40--43]{LubotzkyMagid}, for
representations of Euclidean 2-orbifolds 
in $\mathrm{GL}_2(\CC)$ that are proved to be non-reduced. Furthermore
in \cite[\S 9.3]{GoldmanMillson} Goldman and Millson prove that those are double points,
as the singularities are at most quadratic. We address triple points in Example~\ref{ex:Whitehead} below,
using Nil instead of Euclidean orbifolds.

The fundamental group of  the figure eight knot exterior has a presentation:
\begin{equation*}
\pi_1(S^3\setminus K)=
\langle a,b \mid ab^{-1}a^{-1}b a=b  ab^{-1}a^{-1}b \rangle
\end{equation*}
where $a$ and $b$ are represented by meridian loops. Thus a presentation of the orbifold fundamental group 
can be obtained by adding the relation  $a^3=1$:
\begin{equation*}
\Gamma=\pi_1^{\textrm{orb}}( S^3(K,3) )=
\langle a,b \mid ab^{-1}a^{-1}b a=b  ab^{-1}a^{-1}b, \ a^3=1\rangle. 
\end{equation*}

The algebra $B(\Gamma)$ is a quotient of 
 $B(\mathbb{F}_2) \cong \CC[t_{a},t_{b},t_{ab}]$ by and ideal $I$. 
A computer supported calculation yields \cite{notebook}:
$$
I=(t_{a}-2,t_{b}-2, t_{ab}-2 )
\cap  (t_{a}+1,t_{b}+1, t_{ab}+1)
\cap  (t_{a}+1, t_{b}+1,(t_{ab}-1)^2).
$$
So $X(\Gamma,\SL2)$ consists of three points, one of them non-reduced (double point). The point
$t_{a}=t_b=t_{ab}=2$ is  the trivial character, 
$t_{a}=t_b=t_{ab}=-1$ is the non-trivial abelian character, and 
$(t_{a},t_b,t_{ab})=(-1,-1,1)$ is the \emph{double point} that corresponds
to the character of a simple representation in $\mathrm{SU}(2)$.
This representation is a lift of the rotational part of the 
holonomy of the Euclidean structure in 
$\mathrm{PSU}(2)\cong \mathrm{SO}(3)$.

We can also understand this double point as
the result of a tangency. 
The variety of characters of the figure eight knot exterior
is the plane curve given by
\begin{equation*}
 ( x^{2} y - 2 x^{2}  - y^{2} + y + 1) ( y-x^2+2)=0, 
\end{equation*}
where $x=t_a=t_b$ and $y=t_{ab}$. The group $\Gamma$ is obtained 
from  the figure eight knot
by adding the relation $a^3=1$. 
A representation of  an element of order three is either trivial
(and has trace $2$) or has trace $-1$. Therefore, the case where the image of $a$ is non-trivial corresponds to $x=-1$. The line $x=-1$
intersects $y-x^2+2=0$ transversely 
(at the abelian representation) and $ x^{2} y - 2 x^{2}  - y^{2} + y + 1 =0$ tangentially, giving the double point.

The fact that the intersection between 
$ x^{2} y - 2 x^{2}  - y^{2} + y + 1 =0$ and  $x=-1$
is not transverse 
 corresponds to the fact that
the trace of the meridian is not a local parameter 
at a Euclidean degeneration of hyperbolic cone manifolds \cite{JoanTopology}.
\end{Example}

\begin{Example}
 \label{ex:Whitehead}
Let $S^3(\mathrm{Wh},(m,n))$ be the orbifold 
with underlying space $S^3$, branching locus the 
Whitehead link $\mathrm{Wh}$, and branching indexes $m$ and $n$.
The orbifold $S^3(\mathrm{Wh},(4,2))$
has Nil geometry
\cite[p.~112]{ThesisEva}.
The isometry group of Nil surjects onto 
$\mathrm{Isom}(\mathbb R^2)$, which in its turn surjects onto $ \mathrm{O}(2)\subset \mathrm{SO}(3)$,
but the rotational part of the Nil holonomy of
$S^3(\mathrm{Wh},(4,2))$ 
in $ \mathrm{O}(2)\subset \mathrm{SO}(3) 
\cong \mathrm{PSU}(2)$
does not lift to $\mathrm{SU}(2)$, because of the elements of order two. 
Instead, it lifts to a representation of $S^3(\mathrm{Wh},(8,4))$ in  $\mathrm{SU}(2)$, 
and this is 
the orbifold we consider.

As in the previous example, we notice that
the trace of
the meridian is not a local parameter at a Nil degeneration \cite{Porti_Nil}, it is in fact a singularity of order three.  
Hence this representation will 
be a non reduced point of the character scheme of 
$S^3(\mathrm{Wh},(8,4))$.

For the explicit computation,  start with the fundamental group of the  Whitehead link exterior
$$
\pi_1(S^3\setminus \mathrm{Wh})=
\langle a,b \mid
aba^{-1}b^{-1}a^{-1}ba b = b 
aba^{-1}b^{-1}a^{-1}ba\rangle.
$$
Its scheme of characters is the hypersurface of $\CC^3$ with equation 
$$
(z^3   -xyz^2 +(x^2 + y^2-2)z  - xy)
(x^2+y^2+z^2-xyz-4)=0, 
$$
where $x=t_a$, $y=t_b$ and $z=t_{ab}$. 
The component $x^2+y^2+z^2-xyz-4=0$
consists of abelian characters. The intersection of 
$z^3   -xyz^2 +(x^2 + y^2-2)z  - xy=0$
 with the lines $x=\pm \sqrt 2$ and $y=0$
(corresponding to rotations of order $8$ and $4$
respectively)  is 
$
z^3=0
$. Hence two triple points.

The whole scheme of characters $X(S^3(\mathrm{Wh},(8,4)))$ can be computed using the notebook \cite{notebook}:
besides the two triple points, it contains 21 simple
points ($3$ simple characters, $ 4$ central characters  and $14$ abelian non-central characters).

%
%
%
%
%
%
%
%
%
%
%
%
%
\end{Example}

\begin{Example}
\label{ex:BP} In 
this example and the next one, 
we consider a 3-manifold $M_n$ with boundary a torus obtained
by attaching a $(2, 2n) $-torus link exterior $C_n$ 
 (a cable space) and  $K$, the 
 \emph{orientable} $I$-bundle over the Klein bottle, along
 a boundary component. 
 
 The $I$-bundle $K$ has the homotopy type of the Klein
 bottle, hence
 \[
  \Gamma_K = \pi_1(K) \cong \langle \gamma,\mu\mid \mu\gamma \mu^{-1} =\gamma^{-1}\rangle\,,
 \]
and the peripheral subgroup  is 
$\pi_1(\partial K) = \langle \gamma,\mu^2 \rangle$.

Let $C_n$ be the exterior of the $(2,2n)$-torus link, with $n\geq 2$. 
It  is a Seifert-fibered space with orbifold surface an annulus with one cone point of order $n$.
Hence
 \[
 \Gamma_{C_n} = \pi_1(C_n)\cong
  \langle a,c\mid [a,c^n]=1\rangle ,
 \]
where  $c^n$ is a generator of the center. Up to conjugation $ \pi_1(C_n)$ has two peripheral subgroups:
\(
 \langle ac,c^n\rangle \)  and \(
 \langle a,c^n\rangle
\).

We obtain a $3$-manifold 
$M_n = C_n\cup_h K $ by identifying the boundary component of $K$ 
with a  boundary component of $C_n$ via a homeomorphism
 $h\co \partial K \to \partial_2 C_n$,  where $\partial_2 C_n$
 is the component of $\partial C_n$ with peripheral subgroup $\langle a,c^n\rangle$
 (thus \(
 \langle ac,c^n\rangle \) is the peripheral subgroup of $M_n$).
The homeomorphism $h\colon \partial K \to \partial_2 C_n$ is chosen so that 
$ h_*(\gamma) = a$, and $h_*(\gamma\mu^2)=c^n$.  
A presentation of the fundamental group of $M_n$
is  
\[
\Gamma_n := \pi_1(M_n)
\cong \Gamma_{C_n} *_{\ZZ\oplus\ZZ} \Gamma_K
 \cong  \langle a,c,\gamma,\mu \mid
 \mu\gamma\mu^{-1}\gamma= [a,c^n]=1,\ \gamma\mu^2 = c^n,\ a =\gamma\rangle\,.
\]
We simplify the presentation:
\[
\Gamma_n \cong 
 \langle c,\mu,\gamma \mid 
 \mu\gamma\mu^{-1}=\gamma^{-1},\ \gamma\mu^2 = c^n\rangle
 \cong
 \langle c,\mu \mid 
 \mu  c^n\mu^{-2}\mu^{-1}c^n\mu^{-2}=1\rangle 
 \cong
  \langle c,\mu \mid 
 \   c^n\mu^{-3}c^n\mu^{-1}=1\rangle\,.
\]
The peripheral subgroup of $ \pi_1(M_n)$ is 
$\langle \mu^{-2}c, c^n\rangle$.

Let us consider the case $n=2$. The scheme
$X(M_2,\SL2)$ can be computed from the 
notebook \cite{notebook}, we describe the components. Firstly, since
the abelianization of $\pi_1(M_2)$ is 
$\ZZ\times\ZZ/4\ZZ$, there are three components of 
non-simple (or abelian) characters, described in 
Example~\ref{ex:Z+z4}. In addition there are two components
containing simple characters, given by the ideals
\[ 
 I = (t_{c}t_{c\mu}+2t_\mu, t_\mu^2, t_c t_\mu,\,t_{c}^2)\quad\text{ and } \quad 
 J=(t_\mu,t_{c\mu}) \, .
\]
The ideal $J$ is radical but $I$ is not:
$
\operatorname{rad}(I) = (t_c,t_{\mu})
$.

We check that $\CC[t_c,t_\mu,t_{c\mu}]/I$ 
is the coordinate ring of a double line,  
following 
Section II.3.5 in \cite{EisenbudHarris}, by considering new coordinates:
\[
 x := t_c\,t_{c\mu} + 2\,t_{\mu}\,,\quad
 y :=  t_{c}\,,\quad z:= t_{c\mu}\,.
\]
With these coordinates $I \cong (x,y^2)$ and therefore
\[
 \CC[t_c,t_\mu,t_{c\mu}]/I \cong \CC[x,y,z]/(x,y^2)\cong
 \CC[y,z]/(y^2)\,.
\]
The representation variety (not the scheme) of $M_2$ is studied by K.~Baker and K.~Petersen in \cite{BakerPetersen}. In fact $M_2$ is a 
once-punctured torus bundle with tunnel number
one. This torus bundle is also named $M_2$ in \cite{BakerPetersen}, and in \cite[Section~2.1]{BakerPetersen} the following presentation for 
$\pi_1(M_2)$ is given:
\[ \pi_1(M_2) \cong
 \langle \alpha,\beta \mid  
\beta^{-2} = \alpha^{-1}\beta\alpha^2\beta\alpha^{-1}\rangle\quad
\text{ by putting $\mu = \alpha^{-1}$ and $c = \beta\alpha^{-1}$.}
\]
\end{Example}

\begin{Example} 
\label{ex:BPn}
We continue Example~\ref{ex:BP}: now we show that  
for general $n\geq 2$, the scheme $X(M_n,\SL2)$ is non-reduced, by
generalizing the double line when $n=2$. 
As $\pi_1(M_n)$ is generated by $c$ and $\mu$, $X(M_n,\SL2)_{\textrm{red}}$ is
a subvariety of $\CC^3$ with coordinates $(t_c, t_\mu, t_{c\mu})$. Consider
$$
Y_k=\{
(t_c, t_\mu, t_{c\mu}) \in\mathbb C^3 \mid t_\mu=0, t_c=2\cos (\tfrac{\pi k}{n})
\}\subset X(M_n,\SL2)_{\textrm{red}} .
$$
for $k=1,\ldots , n-1$. 

\begin{Lemma}
\label{Lemma:Ycomponent}  For $k=1,\ldots, n-1$,
$Y_k$ is a component of  $X(M_n,\SL2)_{\textrm{red}}$ containing simple characters.
\end{Lemma}

\begin{Lemma}
\label{Lemma:Ynoreduced}
The scheme  $X(M_n,\SL2)$ is non-reduced at any point of 
$Y_k$, for $k=1,\ldots, n-1$.
%
\end{Lemma}

\begin{proof}[Proof of Lemma~\ref{Lemma:Ycomponent}]
 First we describe  a one-parameter family of characters and next  we prove that this family  forms a component.
 Consider characters
 $\chi\in X(M_n,\SL2)$ such that:
 \begin{itemize}
  \item $\chi(\mu)=\chi(\mu\gamma)=0$ and $\chi(\gamma)=2$. Namely $\chi$ restricted to $\pi_1(K)$
  is the character of a representation of $\pi_1(K)$ that maps $\mu$ to a rotation of angle $\pi$ in $\mathbb H^3$
  and $\gamma$ to the identity.
  \item $\chi(c)=\chi(ac)=\cos(\frac{\pi k}{n})$ and $\chi(a)=2$. Therefore $\chi$ restricted to $\pi_1(C_n)$
  is the character of a representation of $\pi_1(C_n)$ that maps $c$ to a rotation of angle $2\pi k/n$ in $\mathbb H^3$
  and $a$ to the identity.
  \end{itemize}
Thus the restriction of $\chi$ to both $\pi_1(K)$ and $\pi_1(C_n)$ is the character of a representation with image a cyclic group, and the restriction to the 
 attaching torus is central, i.e.~contained in $\{\pm{\id}\}$. 
One can deform such a character by 
 conjugating separately
 the image of $\pi_1(K)$ and the image of $\pi_1(C_n)$, yielding a one-parameter of conjugacy classes of representations of $\pi_1(M_n) $.

To prove that this family of characters is  a component, we first look at representations
of $\pi_1(K)=\langle  \gamma,\mu\mid \mu\gamma \mu^{-1} =\gamma^{-1}\rangle$. They lie in two families: 
\begin{itemize} 
 \item[(a)] Representations that map $\gamma$ to 
 $\pm{\id}$, and $\mu$
 to any value.

 \item[(b)] Representations that preserve an unoriented hyperbolic geodesic 
 $l\subset\mathbb{H}^3$, that map 
$\gamma$ to an isometry preserving the orientation of $l$, and that map $\mu$
and  $\mu\gamma$ to $\pi$-rotations with axis perpendicular to $l$.
%
\end{itemize}
Given $\rho$ a representation of $\pi_1M$, we make two assertions:
\begin{itemize}
 \item If the restriction $\rho\vert_{\pi_1K}$ is in case (a) and $\rho(\mu^2)\neq -{\id}$,
 then $\rho(\pi_1M)$ is abelian.
 \item If the restriction $\rho\vert_{\pi_1K}$ is in case (b) and  $\rho(\gamma)\neq\pm {\id}$, then 
 $\rho(\pi_1M)$ preserves an unoriented geodesic in hyperbolic space.
\end{itemize}
Both assertions follow 
from 
the attaching relations 
$\gamma\mu^2= c^n$ and $a=\gamma$ by elementary considerations, and they yield that $Y_k$ is a component.
\end{proof}


\begin{proof}[Proof of Lemma~\ref{Lemma:Ynoreduced}]
For any  $\chi\in Y_k$ simple,  $\chi$ restricted to $C_n$ is constant (see the proof of Lemma~\ref{Lemma:Ycomponent}). 
In particular its restriction
to $\partial M_n$ is also constant. By writing $\chi=\chi_\rho$, as we can construct deformations of $\chi=\chi_\rho$ that remain constant in 
$\partial M_n$, 
$$
\dim \big( \ker\big(H^1(M_n, \Ad\rho)\to H^1(\partial M_n, \Ad\rho)\big)\big)\geq 1.
$$
On the other hand, by a standard argument on Poincar\'e duality \cite{Hodgson,Sikora}, 
$$
\operatorname{rank}\big(H^1(M_n, \Ad\rho)\to H^1(\partial M_n, \Ad\rho)\big)= 1
$$
Thus $\dim H^1( M_n; \Ad\rho)\geq 2$.
It follows that the Zariski tangent space to the scheme at a simple $\chi\in Y_k$ has dimension $\geq 2$ 
by Theorem~\ref{thm:H1ZT}. 
As generic characters in  $Y_k$ are simple and 
$\dim Y_k=1$, the scheme is non-reduced 
at generic points of $Y_k$.
By Lemma~\ref{lem:non-reduced_closed}, the set of non-reduced points is the Zariski-closed and it contains $Y_k$.
\end{proof}

\end{Example}

\begin{Remark}
The line $Y_k$ restricts to a point in $X(\partial M_ n,\SL2)$. 
The image of the restriction is always an isotropic subspace, and when the 
restriction is non-singular the image is a Lagrangian submanifold, cf.~\cite{Hodgson,Sikora}.
This example proves that non-singularity is needed to have a Lagrangian submanifold. 
\end{Remark}

\begin{Remark}
 The manifold $M_n$ of Example~\ref{ex:BPn} 
 should be compared to an example of Schnauel and Zhang
 \cite{SchnauelZhang}, 
 who attach a cable space to a torus knot exterior. 
 In this way they obtain a boundary slope not detected
 by valuations on the variety of characters.
 \end{Remark}

\section{The character scheme of the Borromean rings}

Let us start with a presentation for the group $\Gamma_\mathit{Bor}$ of the Borromean rings.

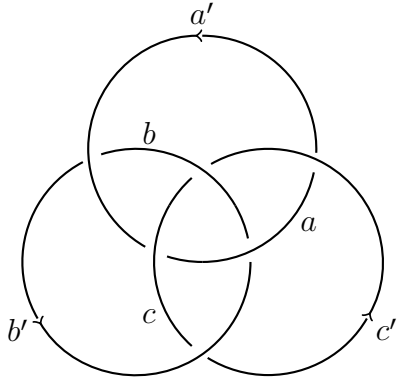
\begin{figure}[htbp]
\begin{center}
\begin{tikzpicture}
\draw[thick, >-] (0,2.5)  arc[ radius = 1.5, start angle= 90, end angle= 240];
\draw[thick] (0,2.5)  arc[ radius = 1.5, start angle= 90, end angle= -2];
\draw[thick] (0 , -.5) arc[ radius = 1.5, start angle= 270, end angle= 252];
\draw[thick] (0 , -.5) arc[ radius = 1.5, start angle= 270, end angle= 360-12];
\draw(0,2.8) node{$a'$};
\draw(1.4,0) node{$a$};
\begin{scope}[rotate=120]
\draw[thick, >-] (0,2.5)  arc[ radius = 1.5, start angle= 90, end angle= 240];
\draw[thick] (0,2.5)  arc[ radius = 1.5, start angle= 90, end angle= -2];
\draw[thick] (0 , -.5) arc[ radius = 1.5, start angle= 270, end angle= 252];
\draw[thick] (0 , -.5) arc[ radius = 1.5, start angle= 270, end angle= 360-12];
\draw(0,2.8) node{$b'$};
\draw(1.4,0) node{$b$};
\end{scope}
 \begin{scope}[rotate=-120]
\draw[thick, >-] (0,2.5)  arc[ radius = 1.5, start angle= 90, end angle= 240];
\draw[thick] (0,2.5)  arc[ radius = 1.5, start angle= 90, end angle= -2];
\draw[thick] (0 , -.5) arc[ radius = 1.5, start angle= 270, end angle= 252];
\draw[thick] (0 , -.5) arc[ radius = 1.5, start angle= 270, end angle= 360-12];
\draw(0,2.8) node{$c'$};
\draw(1.4,0) node{$c$};
\end{scope}
\end{tikzpicture}
\end{center}
\caption{The Borromean rings}
\label{fig}
\end{figure}

We take the over presentation and we obtain:
\[
\left\{
\begin{array}{ccc}
a' & = & c' a (c')^{-1},\\
b' & = & a' b (a')^{-1},\\
c' & = & b' c (b')^{-1},
\end{array}\right.
\quad \text{ and }\quad
\left\{
\begin{array}{ccc}
a' & = & c a c^{-1},\\
b' & = & a b a^{-1},\\
c' & = & b c b^{-1}.
\end{array}\right.
\]
This gives the presentation
\begin{equation}\label{eq:pres_bor}
\Gamma_\mathit{Bor}\cong
\langle a, b, c \mid [a,[b,c^{-1}]],\ [b,[c,a^{-1}]], [c,[a,b^{-1}]] \rangle
\end{equation}
and as usual one of the relation is a consequence of the other two relations.
The generators $a$, $b$ and $c$ are meridians, and the corresponding longitudes are 
\[ 
\ell_a = [b,c^{-1}], \quad \ell_b= [c,a^{-1}] \quad \text{ and } \quad \ell_c =[a,b^{-1}]\,.
\]
In what follows we consider the free group $\mathbb F(a,b,c)$ and the surjection 
$\mathbb F(a,b,c)\twoheadrightarrow\Gamma_{Bor}$. Notice also that there is a surjective homomorphism
$\Gamma_\mathit{Bor}\twoheadrightarrow\mathbb{Z}^3$ (the abelianization). Moreover, there are three obvious surjections $\Gamma_\mathit{Bor}\twoheadrightarrow \mathbb  F_2$, one of the three meridians is mapped onto the trivial element the other meridians are mapped to a generating pair of $\mathbb F_2$.

From these surjections, we obtain an closed immersions
\[
X(\mathbb{Z}^3,\SL2),\ X(\mathbb F_2,\SL2) \hookrightarrow X_{Bor} \hookrightarrow 
X(\mathbb F(a,b,c),\SL2)\,.
\]
Here we write $X_\mathit{Bor} : = X(\Gamma_\mathit{Bor},\SL2)$ for short.

Recall that the coordinate ring of $X(F(a,b,c),\SL2)\subset\CC^7$ is reduced and isomorphic to the quotient
\[
\CC[t_a,t_b,t_c,t_{ab},t_{ac},t_{bc},t_{abc}]/(t_{abc}^2 - p\, t_{abc} + q)
\]
where $p$ and $q$ are the polynomials defined in equations \eqref{eq:p} and \eqref{eq:q} respectively.

\medskip

The scheme $X_\mathit{Bor}$ has several components. 
These components arise from the surjections of $\Gamma_\mathit{Bor}$
 onto $\mathbb{Z}^3$, and the free group $\mathbb F_2$.
%
%
Let us first identify the components: 
\begin{itemize}
\item First of all there is at least one the distinguished component $X_0$ 
which contains the character of a lift of the holonomy representation. We aim to determine its ideal~$I_0$.

\item There are the characters of the non-simple representations $\ns X\subset X_\mathit{Bor}$. Since 
the Abelianization is a surjection $\Gamma_\mathit{Bor}\twoheadrightarrow \ZZ^3$ 
we obtain an inclusion $X(\ZZ^3,\SL2)\hookrightarrow X_\mathit{Bor}$. 
The ideal $\ns I$
corresponding to this component was investigated in Section~\ref{sec:abelian}.
According to Proposition~\ref{Prop:zbeta}, the ideal $I_{ns}$ is generated by
\begin{multline}
t_a^2 + t_b^2 + t_{ab}^2 - t_a t_b t_{ab} -4,\quad 
t_a^2 + t_c^2 + t_{ac}^2 - t_a t_c t_{ac} -4,\quad 
t_b^2 + t_c^2 + t_{bc}^2 - t_b t_c t_{bc} -4,\\
t_a (t_a t_{bc}+t_b t_{ac}+t_c t_{ab}-t_a t_b t_c)-2 t_{ab} t_{ac}-4 t_{bc}+2 t_b t_c,\\
t_b (t_a t_{bc}+t_b t_{ac}+t_c t_{ab}-t_a t_b t_c)-2 t_{ab} t_{bc}-4 t_{ac}+2 t_a t_c,\\
t_c (t_a t_{bc}+t_b t_{ac}+t_c t_{ab}-t_a t_b t_c)-2 t_{ac} t_{bc}-4 t_{ab}+2 t_a t_b, \\
                2 t_{abc}-t_a t_{bc}-t_b t_{ac}-t_c t_{ab}+t_a t_b t_c. 
\end{multline}
The ideal $I_{ns}$ is a prime, and hence radical, since 
$ \CC[t_a,t_b,t_c,t_{ab},t_{ac},t_{bc},t_{abc}]/I_{ns}$
the coordinate algebra of the irreducible variety
$(\CC^*\times\CC^*\times\CC^*)/\sim$ (see Theorem~\ref{thm:abelian}).

\item 
Next there are representations which map one of the generators to a central element $\pm \id\in\SL2$.
This gives rise to six $3$-dimensional components
$X_a^\pm$, $X_b^\pm$ and $X_c^\pm$
in $X_\mathit{Bor}$. Each of those components is isomorphic to $\CC^3$ and has therefore no singular points. The ideal corresponding to
$X_a^\pm$ is 
\[
I_a^\pm := \Big( t_a  \mp 2, t_{ab} \mp t_b, t_{ac} \mp t_c, t_{abc} \mp t_{bc} \big)
\]
and 
\[
\CC[t_a,t_b,t_c, t_{ab},t_{ac}, t_{bc},t_{abc}]/ I_a^\pm \cong \CC[t_b,t_c, t_{bc}]\,.
\]
Similar for the other generators. All these ideals are prime ideals, and hence they are radical.
\end{itemize}

\begin{Theorem} The coordinate algebra $\CC[X_{Bor}] =\CC[X_{Bor}]_{red}$ is reduced. More precisely,
we have 
\[  
\CC[X_{Bor}] \cong \CC[t_a,t_b,t_c, t_{ab},t_{ac}, t_{bc},t_{abc}] / I_{Bor}
\]
where $I_{Bor} = \ns I \cap I_a^+ \cap I_a^-\cap I_b^+ \cap I_b^-
\cap I_c^+ \cap I_c^- \cap I_0$ is the intersection of prime ideals in
$\CC[t_a,t_b,t_c, t_{ab},t_{ac}, t_{bc},t_{abc}]$. Therefore, the representation scheme has eight irreducible components
\[
X_{Bor} = \ns X \cup X_a^+ \cup X_a^- \cup X_b^+ \cup X_b^-\cup X_c^+ \cup X_c^-\cup X_0\,.
\]
\end{Theorem}
\begin{proof}
Computer supported calculations \cite{notebook} give us an ideal 
$$I_7 =\ns I \cap I_a^+ \cap I_a^-\cap I_b^+ \cap I_b^-
\cap I_c^+ \cap I_c^- \subset \CC[t_a,t_b,t_c, t_{ab},t_{ac}, t_{bc},t_{abc}],$$ 
such that $I_\mathit{Bor} \subset I_7$.

If $I$, $J$ are ideals in $\CC[x_1,\ldots,x_n]$ then the \emph{ideal quotient} is 
\[
 (I : J) = \{ f\in \CC[x_1,\ldots,x_n] \mid gf\in I \text{ for all $g\in J$}\}
\]
(see Chapter~4 in \cite{Cox}).
It turns out that  $(I:J)$ is an ideal in 
$\CC[x_1,\ldots,x_n]$ containing $I$. Moreover we have that
\[
V(I:J) \supset \overline{V(I) \setminus V(J)}
\]
where $V(I)$ denotes the vanishing set of the ideal in $\CC^n$ and 
$\overline{S}$ the Zariski-closure of $S\subset\CC^n$.

Now, we use ideal division to define the ideal
$
I_0 := (I_\mathit{Bor}: I_7)
$. The notebook calculations give us that $I_0$ is radical in $\mathbb{Q}[t_a,t_b,t_c,t_{ab},t_{ac},t_{bc},t_{abc}]$ , and 
$I_{Bor} = I_0 \cap I_7$. Hence
\[
I_\mathit{Bor} = \ns I \cap I_a^+ \cap I_a^-\cap I_b^+ \cap I_b^-
\cap I_c^+ \cap I_c^- \cap I_0\,.
\]
The ideal $I_0$ is generated by
\begin{alignat}{3}
 &t_{abc}^2 - t_{abc}\,p +q, &\qquad & t_at_{bc} - t_bt_{ac}, &\qquad &t_bt_{ac} -t_ct_{ab},\notag\\
&t_at_{abc}-t_{ab}t_{ac}, & \qquad &t_bt_{abc}-t_{ab}t_{bc}, &\qquad &t_ct_{abc}-t_{ac}t_{bc}\,.\label{eq:I_0}
\end{alignat}

\begin{Lemma}\label{lem:bor_red}
The ideal $I_0$ is radical in $\CC[t_a,t_b,t_c,t_{ab},t_{ac},t_{bc},t_{abc}]$.
\end{Lemma}
\begin{proof}
Computer supported calculations show that 
$I_0\subset\mathbb{Q}[t_a,t_b,t_c,t_{ab},t_{ac},t_{bc},t_{abc}]$ is a radical ideal (see \cite{notebook}). Now,  Lemma~3.7 in \cite{HMP} shows that
\[
I_0\cdot \CC[t_a,t_b,t_c,t_{ab},t_{ac},t_{bc},t_{abc}] \subset \CC[t_a,t_b,t_c,t_{ab},t_{ac},t_{bc},t_{abc}]
\]
is also radical. 
\end{proof}

We let $X_0$ denote the corresponding subvariety of $X_\mathit{Bor}$. 
The equations~\eqref{eq:I_0} had already be studied by Sparaco \cite{Sparaco}.
\begin{Lemma}[Sparaco]
The variety $X_0$ is irreducible.
\end{Lemma}
\begin{proof}
The projection on to $\mathbb{C}^7\to\mathbb{C}^4$ onto the subspace generated by the $t_{ab}$, $t_{ac}$, $t_{bc}$ and $t_{abc}$ coordinates induces a birational map between $X_0$ and the hypersurface with equation
\begin{multline*}
t_{abc}^{4} - t_{abc}^2 (2 t_{ab} t_{ac} t_{bc} - t_{ab}^{2} - t_{ac}^{2} - t_{bc}^{2} + 4)\\ 
+ t_{ab}^{2} t_{ac}^{2} t_{bc}^{2} -  t_{ab}^{3} t_{ac} t_{bc} -  t_{ab} t_{ac}^{3} t_{bc} -  t_{ab} t_{ac} t_{bc}^{3} + t_{ab}^{2} t_{ac}^{2} + t_{ab}^{2} t_{bc}^{2} + t_{ac}^{2} t_{bc}^{2}
\end{multline*}
It follows now that this hypersurface is irreducible and hence $X_0$ also (see \cite{Sparaco}).
\end{proof}
It follows that the ideal $I_0$ is a prime ideal.
\end{proof}

\subsection{The distinguished component $X_0$}
Notice that $X_0$ is the unique component containing faithful representations. Hence $X_0$ is the canonical component of the character variety of the  Borromean rings.

The singular locus $X_0^{sing}$ is $1$-dimensional. More precisely,
the $X_0^{sing}$ is the union of twelve lines and six points \cite{notebook}. The lines are formed by characters of non-simple representations. 

The characters of the central representations are part of $X_0^{sing}$. At each character of a central representation three of the lines intersect.  So the lines form a \emph{cube} with corners the characters of a central representations.
Hence, the characters of the central representations are singular points of the singular set.

The characters of two central representations given by 
\[
\begin{cases}
a\mapsto \epsilon_a \id\\
b\mapsto \epsilon_b \id\\
c \mapsto\epsilon_c \id\\ 
\end{cases}
\quad\text{ and }\quad
\begin{cases}
a\mapsto \epsilon_a' \id\\
b\mapsto \epsilon_b' \id\\
c \mapsto\epsilon_c' \id\\ 
\end{cases}
\]
are connected by a line in $X_0^{sing}$ if and only if two of the $\epsilon_x$ and $\epsilon_x'$ are equal.
The line consists of characters of non-simple representations.
For example the character of the trivial representation $(\id,\id,\id)$, and the character of the central representation $(\id,\id,-\id)$, are connected by the line
\[
 L=\{(2,2,z,2,z,z,z)\mid z\in\CC\} \subset X_0\,.
\]

The six isolated points in $X_0^{sing}$ are characters of $\operatorname{SU}(2)$-representations.
More precisely, the points are characters of binary-dihedral representations which map one of the generators to $\pm \id$ and the other two to half-turns about two orthogonal lines. For example:
\[
(\pm2,0,0,0,0,0,0)\quad\text{ is the character of }\quad
\rho\co  
\begin{cases}
a\mapsto \pm \id\\
b\mapsto \big(\begin{smallmatrix}i&0\\0&-i\end{smallmatrix}\big)\\
c \mapsto\big(\begin{smallmatrix}0&i\\i&0\end{smallmatrix}\big)
\end{cases}.
\]
These characters are the \emph{midpoints} of the faces of the cube.

\subsection{Intersections between $X_0$ and the other components}\label{sec:int_comp}
All  components do intersect $X_0$. More precisely:
\begin{itemize}
\item $X_0\cap X^\pm_a$: These intersection form the \emph{faces} of the \emph{cube}.

The intersection $X_0\cap X^\pm_a$ is two-dimensional and isomorphic to $\CC^2$. In fact, 
\[
I_0 + I_a^\pm = \big(t_a\mp2, t_{ab}\mp t_b, t_{bc}\mp t_c, 2t_{bc} - t_bt_c,2t_{abc}\mp t_bt_c \big)
\]
and hence for the coordinate ring 
\[
\CC[X_0\cap X^\pm_a] \cong \CC[t_a,t_b,t_c, t_{ab},t_{ac}, t_{bc},t_{abc}]/ (I_0+I_a^\pm) \cong \CC[t_b,t_c]\,.
\]
Similar argument applies for $X_b^\pm$ and $X_c^\pm$.
Notice that the lines in $X_0^{sing}$ are contained in these intersection. For example:
$X_0\cap X^+_a$ contains the following four characters of central representations
\[
\begin{cases}
a\mapsto  \id\\
b\mapsto \epsilon_b \id\\
c \mapsto\epsilon_c \id\\ 
\end{cases}
\quad\text{ where $\epsilon_a,\epsilon_b\in\{\pm1\}$,}
\]
and the following four lines of characters of non-simple representations
\begin{alignat*}{2}
L_1 & =\{(2,2,z,2,z,z,z)\mid z\in\CC\},\quad & 
L_2&=\{(2,-2,z,-2,z,-z,-z)\mid z\in\CC\},\\
L_3&=\{(2,z,2,z,2,z,z)\mid z\in\CC\},\quad &
L_4&=\{(2,z,-2,z,-2,-z,-z)\mid z\in\CC\}\,.
\end{alignat*}
form a \emph{square}.

\item $X_0\cap X_{\mathrm{red}}$: 
This intersection is formed by the 12 lines of non-simple representations. It is contained in $X_0^{sing}$.
\end{itemize}

\bibliographystyle{plain}
\bibliography{traces}

\end{document}